\documentclass[microtype]{gtpart}     % Basic GT/GTM/AGT style
\agtart
\usepackage{graphicx}
\usepackage[font=small]{caption}
\usepackage[font=small]{subcaption}
\usepackage[numbers]{natbib}
\usepackage{stmaryrd}
\usepackage{tikz-cd}
\usepackage{bm}

\graphicspath{{figures/}}

\usepackage{natbibspacing}
\renewcommand{\bibfont}{\small}

\title{Magnitude homology equivalence of Euclidean sets}

\author{Adrián Doña Mateo}
\givenname{Adrián}
\surname{Doña Mateo}
\address{School of Mathematics\\
University of Edinburgh\\
Edinburgh\\
United Kingdom}
\email{adrian.dona@ed.ac.uk}
\urladdr{https://www.maths.ed.ac.uk/~adona}

\author{Tom Leinster}
\givenname{Tom}
\surname{Leinster}
\address{School of Mathematics\\
University of Edinburgh\\
Edinburgh\\
United Kingdom}
\email{tom.leinster@ed.ac.uk}
\urladdr{https://www.maths.ed.ac.uk/~tl}
\keyword{magnitude, magnitude homology, metric space}
\subject{primary}{msc2010}{18G90}
\subject{secondary}{msc2010}{51F99}
\subject{secondary}{msc2010}{52A99}

\arxivreference{2406.11722}

\theoremstyle{plain}
\newtheorem{theorem}{Theorem}[section]
\newtheorem{lemma}[theorem]{Lemma}
\newtheorem{proposition}[theorem]{Proposition}
\newtheorem{corollary}[theorem]{Corollary}

\theoremstyle{definition}
\newtheorem{definition}[theorem]{Definition}
\newtheorem{remark}[theorem]{Remark}
\newtheorem{remarks}[theorem]{Remarks}
\newtheorem{example}[theorem]{Example}
\newtheorem{examples}[theorem]{Examples}

\let\temp\phi
\let\phi\varphi
\let\varphi\temp

\newcommand{\eps}{\epsilon}

\newcommand{\Met}{\mathbf{Met}}
\newcommand{\N}{\mathbb{N}}
\DeclareMathOperator{\im}{im}

\newcommand{\demph}{\textbf}
\newcommand{\bref}[1]{(\ref{#1})}
\newcommand{\MC}[2]{C_{#1, #2}}
\newcommand{\Mh}[1]{H_{#1}}
\newcommand{\MH}[2]{H_{#1, #2}}
\newcommand{\Pc}[1]{P_{#1}}
\newcommand{\PC}[2]{P_{#1, #2}}
\newcommand{\THIN}[2]{T_{#1, #2}}
\DeclareMathOperator{\intc}{ic}
\DeclareMathOperator{\conv}{conv}
\DeclareMathOperator{\cconv}{\overline{conv}}
\DeclareMathOperator{\aff}{aff}
\DeclareMathOperator{\core}{core}
\DeclareMathOperator{\Vol}{Vol}
\newcommand{\preq}{\preceq}
\newcommand{\sub}{\subseteq}
\renewcommand{\epsilon}{\varepsilon}
\renewcommand{\emptyset}{\varnothing}
\newcommand{\cell}[4]{\put(#1,#2){\makebox(0,0)[#3]{#4}}}

\renewcommand{\theenumi}{\roman{enumi}}

\tikzcdset{every arrow/.append style = -{Computer Modern Rightarrow[width=5pt,length=3pt]}}
\newcommand{\oppair}[4]{%
    \setlength\mathsurround{0pt}%
    \begin{tikzcd}[cramped, sep=scriptsize, ampersand replacement=\&]%
        #1 \ar[r, "#3", shift left] \& #2 \ar[l, "#4", shift left]%
    \end{tikzcd}%
    \setlength\mathsurround{0.8pt}%
}
\newcommand{\oppairu}[2]{%
    \setlength\mathsurround{0pt}%
    \begin{tikzcd}[cramped, sep=small, ampersand replacement=\&]%
        #1 \ar[r, shift left] \& #2 \ar[l, shift left]%
    \end{tikzcd}%
    \setlength\mathsurround{0.8pt}%
}
\newcommand{\parpairu}[2]{%
    \setlength\mathsurround{0pt}%
    \begin{tikzcd}[cramped, sep=small, ampersand replacement=\&]%
        #1 \ar[r, shift left] \ar[r, shift right] \& #2%
    \end{tikzcd}%
    \setlength\mathsurround{0.8pt}%
}

\begin{document}

\begin{abstract}    % type your abstract below
Magnitude homology is an $\R^+$-graded homology theory of metric spaces
that captures information on the complexity of geodesics. Here we address
the question: when are two metric spaces magnitude homology equivalent, in
the sense that there exist back-and-forth maps inducing mutually inverse
maps in homology? We give a concrete geometric necessary and sufficient
condition in the case of closed Euclidean sets.  Along the way, we
introduce the convex-geometric concepts of inner boundary and core, and
prove a strengthening for closed convex sets of the classical theorem of
Carath\'eodory.
\end{abstract}

\maketitle

%%%%%%%%%%%%%%%%%%%%   Start of main body of article

\section{Introduction}
\label{sec:intro}

Magnitude homology is a homology theory of enriched
categories~\cite{Leinster2021}. For ordinary categories, it specialises to
the standard homology of categories, which itself includes group
homology and poset homology. But what has sparked the most interest in
magnitude homology (as catalogued in~\cite{MAGBIB}) is that it provides a
homology theory of metric spaces, taking advantage of Lawvere's insight
that metric spaces can be viewed as enriched categories~\cite{LawvMSG}.

The magnitude homology of metric spaces is a genuinely metric
invariant. For example, whereas topological homology detects the existence
of holes, magnitude homology detects their diameter (Theorem~5.7
of~\cite{Kaneta2021}). Whereas the homology of a topological space is
trivial if it is contractible, the magnitude homology of a metric space is
trivial if it is convex (Corollary~\ref{cor:empty_inner_boundary} below,
originally proved by Kaneta and Yoshinaga and, independently, by Jubin). A
theorem of Asao (Theorem~5.3 of~\cite{AsaoMHG}) states that the second
magnitude homology group of a metric space $X$ is nontrivial if $X$
contains a closed geodesic. Gomi proposes a slogan: `The
more geodesics are unique, the more magnitude homology is trivial'
(\cite{GomiMHG}, p.~5).

The story of magnitude homology began with graphs. Hepworth and
Willerton~\cite{Hepworth2017} defined the magnitude homology of a graph and
established its basic properties, treating graphs as metric spaces in which
the distance between two vertices is the number of edges in a shortest path
between them. Later, Leinster and Shulman extended their definition to a
large class of enriched categories, including metric
spaces~\cite{Leinster2021}. In this work, we suppress the enriched
categorical context, working directly and explicitly with metric spaces.

Magnitude homology is not the first homology theory for metric
spaces. Persistent homology, central to topological data analysis, is also
such a theory. It is natural to compare the two theories, as has been done
by Otter~\cite{Otte} and Cho~\cite{Cho}. Here we just note that magnitude
homology and persistent homology capture quite different information about
a space, and that work is underway to use magnitude homology in the
analysis of networks (Giusti and Menara~\cite{GiMe}).

Like persistent homology, the magnitude homology of metric spaces $X$ is a
\emph{graded} homology theory. There is one group $\MH{n}{\ell}(X)$ for
each integer $n \geq 0$ and real number $\ell \geq 0$, where $\ell$ is to
be regarded as a length scale.

As Hepworth and Willerton pointed out in the introduction
to~\cite{Hepworth2017}, magnitude homology is similar in spirit to Khovanov
homology, a graded homology theory of links. The graded Euler
characteristic of Khovanov homology is the Jones polynomial, and the graded
Euler characteristic of magnitude homology is the invariant of metric
spaces known as magnitude \cite{Leinster2013,MAGBIB,MMSCG}.

Magnitude is the canonical measure of the size of an enriched category. For
ordinary categories, under finiteness hypotheses, the magnitude is the
Euler characteristic of the nerve. In particular, when a finite set is seen
as a discrete category, magnitude is cardinality. For metric spaces,
magnitude is geometrically highly informative. For example, for suitable
spaces $X$, the asymptotics of the magnitude of the rescaled space $tX$ as
$t \to \infty$ are known to determine the Minkowski dimension, volume and
surface area of $X$ (Corollary~7.4 of Meckes~\cite{MeckMDC}, Theorem~1 of
Barcel\'o and Carbery~\cite{BaCa}, and Theorem~2(d) of Gimperlein and
Goffeng~\cite{GiGoMFD}, respectively).

The categorification theorem mentioned, that the Euler characteristic of
magnitude homology is magnitude, only holds for \emph{finite} metric spaces
(Theorem~7.12 of~\cite{Leinster2021}). Although there is currently no
categorification theorem for non-finite spaces, the intention is that
magnitude homology is the categorification of magnitude, and shares with it
the property of capturing important geometric features.

This paper addresses the question: when do two metric spaces have the same
magnitude homology? 

To answer this, we first need to make `same' precise. For any homology
theory of any kind of object, there are at least three possible
meanings. The first is simple: just ask that our objects $X$ and $Y$
satisfy $H_n(X) \cong H_n(Y)$ for all $n \geq 0$. This is generally seen as
too loose a relation, and it is too loose for us too. For example, Roff
provided an example of metric spaces whose first magnitude homology groups
are isomorphic but whose first singular homology groups are not
(Section~4.6 of~\cite{Roff2022}). And in
Remark~\ref{rmks:opposing-maps}\bref{rmk:om-circles} below, we describe two
simple but non-homeomorphic metric spaces $X$ and $Y$ such that
$\MH{n}{\ell}(X) \cong \MH{n}{\ell}(Y)$ for all $n$ and $\ell$.

The second option is quasi-isomorphism: generate an equivalence relation on
spaces by declaring them equivalent if there exists a map between them
inducing an isomorphism in homology. The third is more demanding still:
declare $X$ and $Y$ to be equivalent if there exist maps $\oppairu{X}{Y}$
whose induced maps in homology are mutually inverse.

Here we take the third option, defining metric spaces $X$ and $Y$ to be
\emph{magnitude homology equivalent} if there exist maps $\oppairu{X}{Y}$
whose induced maps $\oppairu{\MH{n}{*}(X)}{\MH{n}{*}(Y)}$ are mutually
inverse for all $n \geq 1$. Our maps of metric spaces are those that are short (equivalently
1-Lipschitz, or contractions or distance-decreasing in the non-strict
sense). When metric spaces are viewed as enriched categories, these are the
enriched functors.

Our main theorem (Theorem~\ref{thm:main}) states that two closed subsets of
$\R^N$ are magnitude homology equivalent if and only if they satisfy an
entirely concrete geometric condition: that their `cores' are isometric. 

The core is easily defined. Two distinct points $x$ and $y$ of a metric
space $X$ are \emph{adjacent} if there is no other point $p$ between them
(that is, satisfying $d(x, p) + d(p, y) = d(x, y)$). The \emph{inner
boundary} of $X$ is the set of all points adjacent to some other
point (Figure~\ref{fig:inner_boundary}(a)). For example, the inner boundary of a closed annulus is its inner
bounding circle. Finally, for $X \sub \R^N$, the \emph{core} of $X$ is
the intersection of $X$ with the closed convex hull of its inner boundary
(Figure~\ref{fig:inner_boundary}(b)).

\begin{figure}
    \centering
    \begin{subfigure}[h]{0.45\textwidth}
        \includegraphics[width=\textwidth]{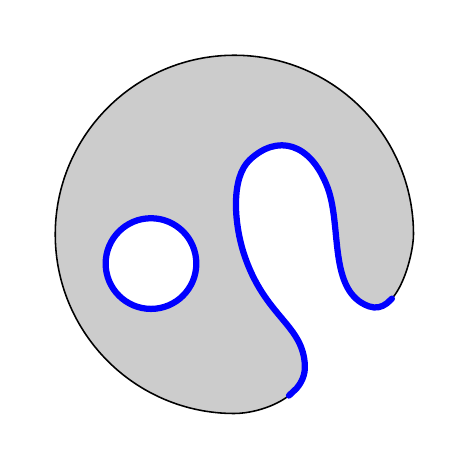}
        \caption{}
    \end{subfigure}
    \hfill
    \begin{subfigure}[h]{0.45\textwidth}
        \includegraphics[width=\textwidth]{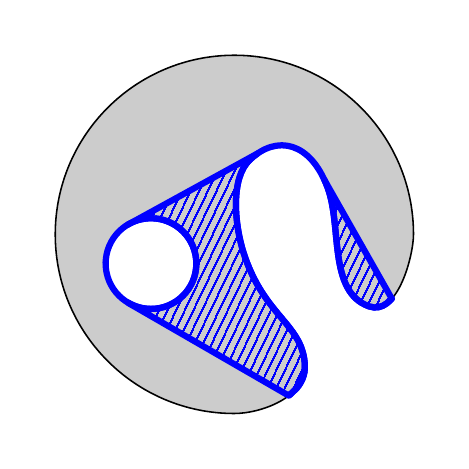}
        \caption{}
    \end{subfigure}

    \caption{(a) A closed subset of $\R^2$, with its inner boundary shown in
    thick blue; (b) the core of the same set, shaded (see
    Section~\ref{sec:core}).}
    \label{fig:inner_boundary}
\end{figure}

In fact, our main theorem says more. Let $X$ and $Y$ be subsets of
Euclidean space. Magnitude homology equivalence of $X$ and $Y$ means that
there exist maps $\oppairu{X}{Y}$ whose induced maps
$\oppairu{\MH{n}{*}(X)}{\MH{n}{*}(Y)}$ are mutually inverse for \emph{all}
$n \geq 1$, but we show that this is also equivalent to them being mutually
inverse for \emph{some} $n \geq 1$.

This startling result is true not because magnitude homology is trivial,
but because $\R^N$ and its subsets are in a certain sense rather simple
metric spaces. For example, any path in $\R^N$ that is locally geodesic is
globally geodesic. Magnitude homology reflects the complexity of geodesics
in a space, so it is unsurprising that its behaviour on subsets of $\R^N$
is rather simple too. 

For arbitrary metric spaces, magnitude homology can be much more complex,
as has been thoroughly established in the case of graphs. For example,
the magnitude homology groups of a subset of $\R^N$ are all free abelian,
whereas Sazdanovic and Summers showed that every finitely generated abelian
group arises as a subgroup of some magnitude homology group of some graph
(Theorem~3.14 of~\cite{SaSu}).  The intrinsic complexity of the magnitude
homology of graphs has been further analysed in recent work of Caputi and
Collari~\cite{CaCo}.

The structure of this paper is as follows. In Section~\ref{sec:aligned}, we
introduce \emph{aligned} metric spaces, which are those in which the
betweenness relation behaves as in subspaces of $\R^N$. This is the most
general class of metric spaces that we consider. We show that alignedness
is equivalent to the conjunction of two properties studied previously:
being geodetic and having no 4-cuts.

The convex-geometric parts of this work require a strengthening of
Carath\'eodory's classical theorem in the case of closed sets
(Section~\ref{sec:cara}). With that in hand, we study inner boundaries and
cores (Sections~\ref{sec:ib} and~\ref{sec:core}). For example, the closed
Carath\'eodory theorem is used to prove a fundamental result: any point in
the convex hull of a closed set $X \sub \R^N$ is either in $X$ itself or
in the convex hull of its inner boundary
(Proposition~\ref{propn:conv-union}). This in turn is used to show that if
$X$ is not convex then every point of $X$ has a unique closest point in its
core---even though the core is not in general convex
(Proposition~\ref{propn:core-retract}).

We then review the magnitude homology of metric spaces, from the beginning
(Section~\ref{sec:mhms}). The magnitude homology of aligned spaces
(Section~\ref{sec:mhst}) is vastly simpler than the general case, thanks to
the structure theorem of Kaneta and Yoshinaga
(Theorem~\ref{thm:thin_chains}). For instance, this theorem implies that
two maps of aligned spaces $f, g: X \to Y$ that agree on the inner boundary
of $X$ induce the same map in homology in positive degree. We improve
slightly on their result, proving that the chain maps induced by $f$ and
$g$ are chain homotopic (Theorem~\ref{thm:homotopic}). But more
importantly, Kaneta and Yoshinaga's theorem leads to a concrete geometric
criterion for when two maps $\oppairu{X}{Y}$ induce mutually inverse maps
in magnitude homology (Theorem~\ref{thm:op-tfae}).

In Section~\ref{sec:mheq}, we introduce magnitude homology equivalence and
prove that every closed, nonconvex subset of $\R^N$ is magnitude homology
equivalent to its core (Theorem~\ref{thm:equiv-to-core}). Since taking the
core is an idempotent process, this theorem provides a canonical
representative for each magnitude homology equivalence class of closed
Euclidean sets. And it is a crucial ingredient in our main theorem
(Section~\ref{sec:main}), which gives several necessary and sufficient
conditions for two closed Euclidean sets to be magnitude homology
equivalent, one of which is that their cores are isometric. We also provide
examples to show that for magnitude homology, the three notions of
homological sameness discussed above are genuinely different.

\section{Aligned spaces}
\label{sec:aligned}

Throughout this work, a \emph{map} of metric spaces means one that is short in the following sense.

\begin{definition}
\label{defn:map}
Let $X$ and $Y$ be metric spaces. A function $f: X \to Y$ is a
\demph{short map}, or just a \demph{map}, if $d(f(x), f(x'))
\leq d(x, x')$ for all $x, x' \in X$.
\end{definition}

When metric spaces are viewed as enriched categories, these are the
enriched functors.

An \demph{isometry} is a map that is distance-preserving, i.e.\ $d(f(x), f(x')) =
d(x, x')$, but need not be surjective.

For the rest of this section, let $X$ be a metric space. 

\begin{definition}\label{def:intervals}
    Let $x, y, z \in X$. We say that $y$ is \demph{between} $x$ and $z$,
    and write $x \preq y \preq z$, if $d(x,z) = d(x,y) + d(y,z)$. If also
    $x \neq y \neq z$, then $y$ is \demph{strictly between} $x$ and $z$,
    written as $x \prec y \prec z$.

    For $a, b \in X$, we define the \demph{closed interval}
    \[
    [a, b] = \{ x \in X : a \preceq x \preceq b \}.
    \]
    The intervals $(a, b)$, $[a, b)$ and $(a, b]$ are defined similarly.
\end{definition}

We will use two elementary facts without mention: first, that $[a, b]$ is
topologically closed in $X$, and second, that if $x \preq y \preq y'$ and
$x \preq y' \preq y$ then $y = y'$.

\begin{definition}\label{def:aligned}
    The space $X$ is \demph{aligned} if for all $n \geq 1$ and $x_0,
    \ldots, x_n \in X$ satisfying $x_{i - 1} \prec x_i \prec x_{i+1}$
    whenever $0 < i < n$, we have
    \[[x_0,x_n] = [x_0,x_1] \cup [x_1,x_2] \cup \cdots \cup [x_{n-1},x_n].\]
\end{definition}

For example, Euclidean space $\R^N$ is aligned. Any subspace of an aligned
space is also aligned.

We view graphs (taken to be connected and undirected) as metric spaces
as follows: the points are the vertices, and the distance between two
vertices is the number of edges in a shortest path connecting them.

\begin{examples}
\label{egs:aligned-graphs}
\begin{enumerate}
\item 
Any tree, seen as a metric space, is aligned. 

    \item
    \label{eg:sg-non}
    None of the graphs
    \[
        \begin{array}{c}
        \begin{tikzpicture}[xscale=-1, rotate=-18, scale=0.8]
            \node (v) at (0:1) {$v$};
            \node (w) at (72:1) {$w$};
            \node (x) at (144:1) {$x$};
            \node (y) at (216:1) {$y$};
            \node (z) at (-72:1) {$z$};
            \draw[-] (v) -- (w) -- (x) -- (y) -- (z) -- (v);
        \end{tikzpicture}
        \end{array}
        \qquad
        \begin{array}{c}
        \begin{tikzpicture}[scale=0.5]
            \node (w) at (-1,1) {$w$};
            \node (x) at (1,1) {$x$};
            \node (y) at (1,-1) {$y$};
            \node (z) at (-1,-1) {$z$};
            \draw[-] (w) -- (x) -- (y) -- (z) -- (w) -- (y);
        \end{tikzpicture}
        \end{array}
        \qquad
        \begin{array}{c}
        \begin{tikzpicture}[scale=0.65]
            \node (w) at (-1,1) {$w$};
            \node (x) at (1,1) {$x$};
            \node (y) at (1,-1) {$y$};
            \node (z) at (-1,-1) {$z$};
            \node (v) at (0, 0) {$v$};
            \draw[-] (w) -- (x) -- (y) -- (z) -- (w) -- (v) -- (y);
        \end{tikzpicture}
        \end{array}
    \]
    is aligned. In the first, $v \prec w \prec x$ and $w \prec x \prec y$,
    but
    \[
    [v, y] = \{v, z, y\} 
    \neq 
    \{v, w, x, y\} = [v, w] \cup [w, x] \cup [x, y].
    \]
    In both the second and the third, $x \prec y \prec z$ but $[x, y] \cup
    [y, z]$ is a proper subset of $[x, z]$.
\end{enumerate}
\end{examples}

We will use without mention the fact that in an aligned space, the equation
\[
[x_0, x_2] = [x_0, x_1] \cup [x_1, x_2]
\]
holds not just when $x_0 \prec x_1 \prec x_2$, but also, slightly more
generally, when $x_0 \preq x_1 \preq x_2$.

\begin{lemma}
\label{lemma:int-sub}
For points $a, b, x, y$ of an aligned space, if $x, y \in [a, b]$ then $[x,
y] \sub [a, b]$.
\end{lemma}

The alignedness condition cannot be dropped. For example, the third
graph of Example~\ref{egs:aligned-graphs}\bref{eg:sg-non} has $w, y
\in [x, z]$ but $v \in [w, y] \setminus [x, z]$. 

\begin{proof}
Let $x, y \in [a, b]$. Then $[a, b] = [a, y] \cup [y, b]$, so without loss
of generality, $x \in [a, y]$. But then $[a, y] = [a, x] \cup [x, y]$, so
$[x, y] \sub [a, y] \sub [a, b]$.
\end{proof}

Thus, intervals $[a, b]$ in an aligned space are convex in the following
sense. 

\begin{definition}
A subset $A$ of $X$ is \demph{convex} in $X$ if $[x,y] \subseteq A$
whenever $x, y \in A$. The \demph{convex hull} $\conv(A)$ of $A$ in $X$ is
the intersection of all convex subsets of $X$ containing $A$.  Its
\demph{closed convex hull} $\cconv(A)$ is the intersection of all closed
convex subsets of $X$ containing $A$.
\end{definition}

Since an arbitrary intersection of convex sets is convex, $\conv(A)$ is the
smallest convex set containing $A$, and similarly for $\cconv(A)$.

\begin{remark}
\label{rmk:L}
In $\R^N$, the closure of a convex set is closed, so the closed convex hull
is the closure of the convex hull. Both these statements are false in an
arbitrary aligned space. For example, take $X$ to be the \textsf{L}-shaped
subspace
\[
\bigl[ (0, 0), (1, 0) \bigr]
\cup
\bigl[ (0, 0), (0, 1) \bigr]
\]
of $\R^2$, and take $A = ((0, 0), (1, 0)] \cup \{(0, 1)\}$.  Then $A$ is
convex in $X$, but its closure is not. In fact, the closed convex hull of
$A$ in $X$ is $X$ itself.
\end{remark}

\begin{remark}
\label{rmk:ic}
In an arbitrary metric space $X$, convex hulls can be constructed as
follows. For $A \sub X$, let $\intc(A)$ denote the \demph{interval closure}
of $A$, defined as
\[
\intc(A) = \bigcup_{x,y \in A} [x,y].
\]
Then
\[
\conv(A) = \bigcup_{n \geq 0} \intc^n(A), 
\]
since the right-hand side is the smallest convex set containing $A$.
\end{remark}

Intervals in aligned spaces embed into the real line:

\begin{lemma}
\label{lemma:intervals-embed}
For points $a$ and $b$ of an aligned space, the function $d(a, -): [a, b]
\to \R$ is an isometry. In particular, points of $[a, b]$ the same distance
from $a$ are equal.
\end{lemma}

\begin{proof}
Let $x, y \in [a, b]$. We must prove that $\left| d(a, y) - d(a, x) \right|
= d(x, y)$. 

By alignedness, $[a, b] = [a, x] \cup [x, b]$, so $y \in [a, x]$ or $y \in
[x, b]$. Similarly, $x \in [a, y]$ or $x \in [y, b]$.

If $y \in [x, b]$ and $x \in [y, b]$ then $x = y$ and the result holds
trivially. Otherwise, without loss of generality, $x \not\in [y, b]$, so $x
\in [a, y]$. Hence $d(a, y) = d(a, x) + d(x, y)$ and the result follows.
\end{proof}

The rest of this section relates alignedness to two conditions on metric
spaces that appear in the magnitude homology literature (and were
generalised from graph theory): being geodetic and having no 4-cuts. We
will need this relationship only to show that a theorem of Kaneta and
Yoshinaga applies to aligned spaces (see Theorem~\ref{thm:thin_chains}
below). The reader willing to take this on trust can omit the rest of this
section.

\begin{definition}\label{def:geodetic}
    The metric space $X$ is \demph{geodetic} if whenever $a, b \in X$ and
    $x, y \in [a, b]$, then either (i)~$a \preq x \preq y$ and $x \preq
    y \preq b$, or (ii)~$a \preq y \preq x$ and $y \preq x \preq b$.
\end{definition}

\begin{definition}
    A \demph{4-cut} in $X$ is a 4-tuple $(a, x, y, b)$ of points such that
    $a \prec x \prec y$ and $x \prec y \prec b$, and yet $x, y$ are not
    both in $[a, b]$. The space $X$ has \demph{no 4-cuts} if no such tuple
    exists.
\end{definition}

\begin{examples}
\begin{enumerate}
\item
Every subspace of $\R^N$ is geodetic with no 4-cuts.

\item 
The first graph of Example~\ref{egs:aligned-graphs}\bref{eg:sg-non} is
geodetic, but has a 4-cut $(v, w, x, y)$. 

\item
The second graph of Example~\ref{egs:aligned-graphs}\bref{eg:sg-non} has
no 4-cuts but is not geodetic, since $w, y \in [x, z]$ but $w \not\in [x,
y]$ and $y \not\in [x, w]$. 

\item
The third graph of Example~\ref{egs:aligned-graphs}\bref{eg:sg-non} has a
4-cut $(x, w, v, y)$ and is not geodetic (again, consider $w, y \in [x, z]$).
\end{enumerate}
\end{examples}

\begin{proposition}\label{prop:aligned-eqv}
    A metric space is aligned if and only if it is geodetic and has no 4-cuts.
\end{proposition}

\begin{proof}
First suppose that $X$ is aligned. That $X$ is geodetic follows from
Lemma~\ref{lemma:intervals-embed} and the fact that subspaces of $\R$ are
geodetic. To prove that $X$ has no 4-cuts, note that whenever $a \prec x
\prec y$ and $x \prec y \prec b$, alignedness gives $[a,b] = [a,x] \cup
[x,y] \cup [y,b]$, hence $x, y \in [a,b]$.

Conversely, suppose that $X$ is geodetic and has no 4-cuts, and take points 
$x_0, \ldots, x_n$ satisfying $x_{i - 1} \prec x_i \prec x_{i + 1}$
whenever $0 < i < n$. We must show that $[x_0, x_n] = \bigcup_{i = 1}^n
[x_{i - 1}, x_i]$. We use induction on $n$. 

The case $n = 1$ is trivial.  For $n = 2$, let $y \in [x_0,x_2]$. Applying
the definition of geodetic to $y, x_1 \in [x_0, x_2]$ gives $y \in
[x_0,x_1]$ or $y \in [x_1,x_2]$. Hence $[x_0, x_2] \subseteq [x_0, x_1]
\cup [x_1, x_2]$. The opposite inclusion holds in any metric space
(Lemma~4.13 of Leinster and Shulman~\cite{Leinster2021}).

Now let $n \geq 3$. No 4-cuts gives $x_0 \prec x_2 \prec x_3$. Hence the
list of points $x_0, x_2, x_3, \ldots, x_n$ satisfies the conditions of
Definition~\ref{def:aligned}, so by inductive hypothesis,
\[[x_0,x_n] = [x_0,x_2] \cup [x_2,x_3] \cup \cdots \cup [x_{n-1},x_n].\]
Finally, $[x_0,x_2] = [x_0,x_1] \cup [x_1,x_2]$ by the $n = 2$ case,
completing the proof.
\end{proof}

\section{A Carath\'eodory theorem for closed sets}
\label{sec:cara}

The classical Carath\'eodory theorem states that any point in the convex
hull of a subset $X$ of Euclidean space must, in fact, lie in the convex
hull of some affinely independent subset of $X$. (See Theorem~1.1.4
of~\cite{Schneider2013}, for instance.) We will need a strengthening of
this result for \emph{closed} sets. It is very classical in flavour, but we
have been unable to find it in the literature.

\begin{theorem}
\label{thm:closed-cara}
Let $X$ be a closed subset of $\R^N$ and $a \in \conv(X)$. Then there exist
$n \geq 0$ and affinely independent points $x_0, \ldots, x_n \in X$ such
that $a \in \conv\{x_0, \ldots, x_n\}$ and
\begin{equation}
\label{eq:cara-int}
\conv\{x_0, \ldots, x_n\} \cap X = \{x_0, \ldots, x_n\}.
\end{equation}
\end{theorem}

The condition that $X$ is closed cannot be dropped: consider $X = (-2, -1)
\cup (1, 2) \sub \R$ and $a = 0$.

\begin{proof}
First assume that $X$ is compact. The result is trivial when $N = 0$, so
let $N \geq 1$ and assume inductively that it holds for $N - 1$. If $a$ can
be expressed as a convex combination of elements of $X$ that all lie in
some proper affine subspace $H$ of $\R^N$, we can apply the inductive
hypothesis to $H \cap X$ and the result is proved. Assuming otherwise,
Carath\'eodory's theorem implies that $a$ is in the convex hull of some
affinely independent subset of $X$, which must then have $N + 1$ elements.

Now write 
\[
\Delta^N = 
\Bigl\{
(p_0, \ldots, p_N) \in \R^{N + 1} : p_i \geq 0, \sum p_i = 1
\Bigr\},
\]
and consider the maps
\[
\begin{array}{ccccc}
X^{N + 1}       &\xleftarrow{\ \pi\ }   &
X^{N + 1} \times \Delta^N       &
\xrightarrow{\ \sigma\ }        &
\R^N    \\
(x_0, \ldots, x_N)      &\longmapsfrom  &
(x_0, \ldots, x_N, p_0, \ldots, p_N)    &
\longmapsto     &
\sum_{i = 0}^N p_i x_i.
\end{array}
\]
Both maps are continuous and $X^{N + 1} \times \Delta^N$ is compact, so the
set
\[
K = 
\bigl\{
(x_0, \ldots, x_N) \in X^{N + 1}: a \in \conv\{x_0, \ldots, x_N\}
\bigr\}
=
\pi(\sigma^{-1}(a))
\]
is compact, as well as nonempty. Moreover, the map
\[
\begin{array}{ccc}
(\R^N)^{N + 1}          &\to        &
\R     \\
(z_0, \ldots, z_N)      &\mapsto    &
\Vol_N (\conv\{z_0, \ldots, z_N\})
\end{array}
\]
is continuous, where $\Vol_N$ denotes $N$-dimensional volume. Hence
$\Vol_N(\conv\{x_0, \ldots, x_N\})$ attains a minimum at some point
$(x_0, \ldots, x_N)$ of $K$. Since $a \in \conv\{x_0, \ldots, x_N\}$, the
points $x_0, \ldots, x_N$ do not lie in any proper affine subspace of
$\R^N$, so $\Vol_N(\conv\{x_0, \ldots, x_N\}) > 0$. 

We now prove equation~\eqref{eq:cara-int} for $x_0, \ldots, x_N$.
Certainly $\{x_0, \ldots, x_N\} \sub \conv\{x_0, \ldots, x_N\} \cap
X$. Conversely, let $x \in \conv\{x_0, \ldots, x_N\} \cap X$. We have
\[
\conv\{x_0, \ldots, x_N\}
=
\bigcup_{i = 0}^N 
\conv\{x_0, \ldots, x_{i - 1}, x, x_{i + 1}, \ldots, x_N\},
\]
so without loss of generality, $a \in \conv\{x, x_1, \ldots, x_N\}$. Then
$(x, x_1, \ldots, x_N) \in K$, so
\begin{equation}
\label{eq:cara-ineq}
\Vol_N(\conv\{x, x_1, \ldots, x_N\})
\geq
\Vol_N(\conv\{x_0, x_1, \ldots, x_N\})
\end{equation}
by minimality. But since $x \in \conv\{x_0, \ldots, x_N\}$,
\begin{equation}
\label{eq:cara-inc}
\conv\{x, x_1, \ldots, x_N\} \sub \conv\{x_0, x_1, \ldots, x_N\},
\end{equation}
so equality holds in~\eqref{eq:cara-ineq}. This and the fact that
$\Vol_N(\conv\{x_0, \ldots, x_N\}) > 0$ together imply that equality holds
in~\eqref{eq:cara-inc}. In particular, $x_0 \in \conv\{x, x_1, \ldots,
x_N\}$, and a short calculation using affine independence of $x_0, x_1,
\ldots, x_N$ then yields $x = x_0$. Hence $x \in \{x_0, \ldots, x_N\}$, as
required. 

This proves the theorem when $X$ is compact. Now consider the general case
of a closed set $X$. Since $a \in \conv(X)$, we have $a \in \conv(F)$ for
some finite $F \sub X$. Then $\conv(F)$ is compact, so $X' = \conv(F) \cap
X$ is compact with $a \in \conv(X')$. So by the compact case, there are
affinely independent points $x_0, \ldots, x_n \in X'$ such that $a \in
\conv\{x_0, \ldots, x_n\}$ and $\conv\{x_0, \ldots, x_n\} \cap X' = \{x_0,
\ldots, x_n\}$. Now $x_0, \ldots, x_n \in X' \sub \conv(F)$, so
$\conv\{x_0, \ldots, x_n\} \sub \conv(F)$, giving
\[
\conv\{x_0, \ldots, x_n\} \cap X
=
\conv\{x_0, \ldots, x_n\} \cap X'
=
\{x_0, \ldots, x_n\}.
\qedhere
\]
\end{proof}

\section{Inner boundaries}
\label{sec:ib}

Our main theorem relates magnitude homology equivalence to two
concrete geometric constructions: the inner boundary and the core. We
introduce them in the next two sections. 

\begin{definition}
   Two points of a metric space are \demph{adjacent} if they are distinct
   and there is no point strictly between them.
\end{definition}

For example, when a (connected) graph is viewed as a metric space, two
vertices are adjacent in this sense if and only if there is an edge between
them. 

\begin{definition}\label{def:inner_boundary}
    The \demph{inner boundary} $\rho X$ of a metric space $X$ is the subset
    \[\rho X = \{ x \in X : x \text{ is adjacent to some point in } X \}.\]
\end{definition}

Figure~\ref{fig:inner_boundary}(a) and the following examples shed light
on the choice of terminology.

\begin{examples}\label{egs:boundary}
    \begin{enumerate}
    \item 
    \label{eg:bdy-annulus}
    For a closed annulus in $\R^2$, the inner boundary is the inner
    bounding circle.

    \item 
    \label{eg:bdy-removed}
    Similarly, when $X$ is $\R^2$ with several disjoint open discs
    removed, the inner boundary is the union of their bounding circles.

    \item 
    \label{eg:bdy-finite}
    The inner boundary of a finite metric space $X$ with more than
    one point is $X$ itself. 
    
    \item 
    \label{eg:bdy-L}
    Let $X$ be the \textsf{L}-shaped space of Remark~\ref{rmk:L}. Then
    $\rho X = X \setminus \{(0,0)\}$. This shows that $\rho X$ need not be
    closed in $X$.
    
    \item 
    \label{eg:bdy-menger}
    A metric space $X$ is said to be \demph{Menger convex} if $\rho X
    = \emptyset$. Every convex or open subset of $\R^N$ is Menger
    convex. For closed sets in $\R^N$, Menger convexity is equivalent to
    convexity in the usual sense (Theorem~2.6.2 of~\cite{Papa}). 
    \end{enumerate}
\end{examples}

\begin{remark}
\label{rmk:bdys}
For $X \subseteq \R^N$, the inner boundary $\rho X$ is a subset of the
topological boundary $\partial X$. But whereas $\partial$ is defined for
subspaces of a topological space, $\rho$ is defined for abstract metric
spaces. Another difference between $\rho$ and $\partial$ is that $\rho$ is
idempotent: $\rho \rho X = \rho X$ for all metric spaces $X$, as is easily
shown.
\end{remark}

\begin{proposition}
\label{propn:conv-union}
    Let $X$ be a closed subset of $\R^N$. Then 
    \[\conv(X) = X \cup \conv(\rho X) \quad \text{and} \quad \cconv(X) = X
    \cup \cconv(\rho X).\] 
\end{proposition}

\begin{proof}
The second equation follows from the first by taking closures, and one
inclusion of the first equation is immediate. It remains to prove that
every point $a$ of $\conv(X)$ is in $X \cup \conv(\rho X)$.

By Theorem~\ref{thm:closed-cara} (closed Carath\'eodory), there exist $n
\geq 0$ and affinely independent points $x_0, \ldots, x_n \in X$ such that
$a \in \conv\{x_0, \ldots, x_n\}$ and
\[
\conv\{x_0, \ldots, x_n\} \cap X = \{x_0, \ldots, x_n\}.
\]
Whenever $i \neq j$, we have $[x_i, x_j] \cap X = \{x_i, x_j\}$, so
$x_i$ is adjacent to $x_j$. Hence if $n \geq 1$ then $x_0, \ldots, x_n \in
\rho X$ and $a \in \conv(\rho X)$, while if $n = 0$ then $a = x_0 \in X$.
\end{proof}

The inner boundary construction is not functorial in the obvious sense. For
example, whenever $Y$ is a closed convex subset of $\R^N$ and $X$ is a
closed but nonconvex subset of $Y$, we have $\rho X \neq \emptyset = \rho
Y$ (by Example~\ref{egs:boundary}\bref{eg:bdy-menger}), so the inclusion $X
\hookrightarrow Y$ cannot induce a map $\rho X \to \rho Y$. However, we now
state conditions under which a map $f: X \to Y$ does restrict to a map
$\rho X \to \rho Y$.

\begin{lemma}\label{lem:restrict_inner_boundary}
    Let $\oppair{X}{Y}{f}{g}$ be maps of metric spaces.
    \begin{enumerate}
        \item \label{part:rib-gf} If $gf(x) = x$ for all $x \in \rho X$,
        then $f$ sends adjacent points of $X$ to adjacent points of $Y$
        and $f(\rho X) \sub \rho Y$;

        \item \label{part:rib-fg} If also $fg(y) = y$ for all $y \in \rho
        Y$, then $f$ and $g$ 
        restrict to mutually inverse isometries $\oppairu{\rho X}{\rho Y}$.
    \end{enumerate}
\end{lemma}

\begin{proof}
For~\bref{part:rib-gf}, suppose that $gf(x) = x$ for all $x \in \rho X$,
and let $x, x' \in X$ be adjacent. We must show that $f(x), f(x') \in Y$
are adjacent. Since $x, x'$ are in $\rho X$, they are fixed by $gf$. Since
$x \neq x'$, it follows that $f(x) \neq f(x')$. (Recall that distinctness
is part of the definition of adjacency.) Now let $y \in [f(x), f(x')]$. By
the triangle inequality, the fact that $gf$ fixes $x$ and $x'$, and
the shortness of $f$ and $g$,
\begin{align*}
d(x, x')        &
\leq
d(x, g(y)) + d(g(y), x')        
=
d(gf(x), g(y)) + d(g(y), gf(x'))        \\
&
\leq
d(f(x), y) + d(y, f(x'))        
=
d(f(x), f(x'))  \\
&
\leq 
d(x, x'),
\end{align*}
so equality holds throughout. Hence 
\[
g(y) \in [x, x'],
\quad
d(x, g(y)) = d(f(x), y),
\quad
d(g(y), x') = d(y, f(x')).
\]
Since $x$ and $x'$ are adjacent, the first statement gives $g(y) \in \{x,
x'\}$, then the others give $y \in \{f(x), f(x')\}$. Hence $f(x)$
and $f(x')$ are adjacent. It follows that $f(\rho X) \sub \rho Y$.

For~\bref{part:rib-fg}, it suffices to note that, by~\bref{part:rib-gf},
$f(\rho X) \subseteq \rho Y$ and $g(\rho Y) \subseteq \rho X$. 
\end{proof}

The rest of this section describes the interactions between the concepts of
inner boundary and convex hull. It is needed for the proof of
Proposition~\ref{prop:core-closure}, but not for the main theorem.

The first two results address the question of whether the convex hull of
points in a set crosses its inner boundary
(Figure~\ref{fig:boundary-crossing}(a,~b)).

\begin{figure}
\begin{center}
\setlength{\unitlength}{1mm}
\begin{picture}(126,50)
\cell{5}{39}{r}{(i)}
\cell{5}{16}{r}{(ii)}
\cell{25}{0}{b}{(a)}
\cell{25}{39}{c}{\includegraphics[width=15\unitlength]{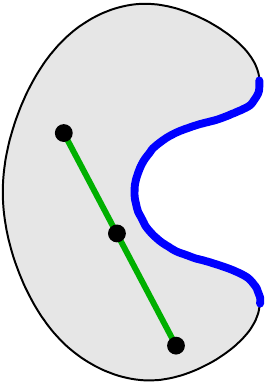}}
\cell{19.5}{42.2}{c}{$x$}
\cell{22}{36.2}{c}{$y$}
\cell{25.5}{30.5}{c}{$z$}
\cell{25}{16}{c}{\includegraphics[width=15\unitlength]{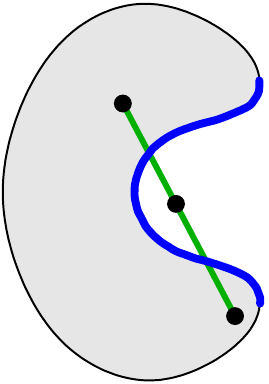}}
\cell{22.5}{21}{c}{$x$}
\cell{29.5}{15}{c}{$y$}
\cell{28.5}{9}{c}{$z$}
\cell{64}{0}{b}{(b)}
\cell{64}{39}{c}{\includegraphics[width=15\unitlength]{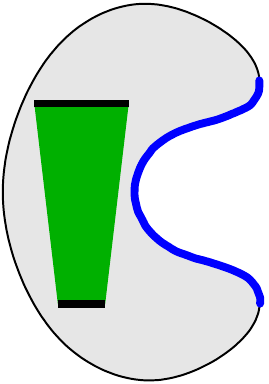}}
\cell{61}{46}{c}{$A$}
\cell{67}{38}{l}{$\conv(A)$}
\put(66.5,38.5){\line(-4,1){5}}
\cell{64}{16}{c}{\includegraphics[width=15\unitlength]{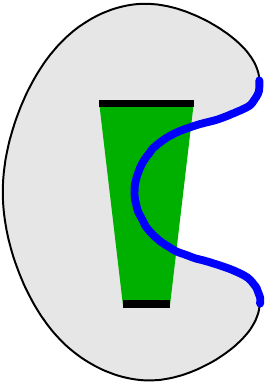}}
\cell{65}{23}{c}{$A$}
\cell{67}{16}{l}{$\conv(A)$}
\cell{103}{0}{b}{(c)}
\cell{103}{39}{c}{\includegraphics[width=15\unitlength]{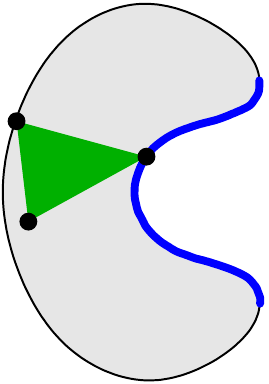}}
\cell{106.5}{40}{c}{$x_0$}
\cell{94}{44}{c}{$x_1$}
\cell{98.5}{35}{c}{$x_2$}
\cell{103}{16}{c}{\includegraphics[width=15\unitlength]{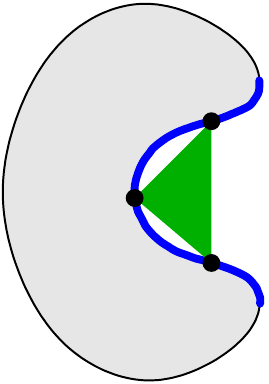}}
\cell{107}{22}{c}{$x_0$}
\cell{100.5}{15}{c}{$x_1$}
\cell{107}{10}{c}{$x_2$}
\end{picture}
\end{center}
\caption{The two cases of
(a)~Lemma~\ref{lem:boundary_in_interval},
(b)~Lemma~\ref{lem:hull_contained}, and
(c)~Proposition~\ref{propn:dichotomy}. Here $X$ is a closed
subset of $E = \R^2$, with inner boundary shown in thick blue. In~(b),
the subset $A$ of $X$ is the union of the two parallel line segments.} 
\label{fig:boundary-crossing}
\end{figure}

\begin{lemma}\label{lem:boundary_in_interval}
    Let $E$ be an aligned metric space in which every closed interval $[a,
    b]$ is compact. Let $X$ be a closed subset of $E$, let $x, z \in X$,
    and let $y \in [x, z]$. Then either (i)~$y \in X$, or (ii)~$[x,y)$ and
    $(y,z]$ both intersect $\rho X$. 
\end{lemma}

\begin{proof}
By hypothesis, $[x,y]$ is compact, and so is $[x,y] \cap X$ because $X$ is
closed; it is also nonempty. Hence it contains a point $y_0 \in [x,y] \cap
X$ that maximises $d(x,y_0)$. Similarly, let $y_1 \in [y,z] \cap X$
maximise $d(y_1, z)$. Assuming that $y \notin X$, we have $y_0 \prec y
\prec y_1$ by Lemma~\ref{lemma:intervals-embed}.

We claim that $y_0$ is adjacent to $y_1$ in $X$. By alignedness and
Lemma~\ref{lemma:intervals-embed},
\[
[y_0,y] \cup [y,y_1] = [y_0, y_1] \sub [x, y],
\]
But $[y_0, y]$ and $[y, y_1]$ can only intersect $X$ at $y_0$ and $y_1$, by
maximality. Hence $[y_0, y_1] \cap X = \{y_0, y_1\}$, as required.
\end{proof}

\begin{remark}
The closed intervals $[a, b]$ in an aligned space need not be compact;
consider $[-1, 1]$ in $\R \setminus \{0\}$, for instance. But
Lemma~\ref{lemma:intervals-embed} implies that intervals $[a, b]$ in an
aligned space $E$ are totally bounded, so they are compact if $E$ is
complete.
\end{remark}

\begin{lemma}\label{lem:hull_contained}
    Let $E$ be an aligned metric space in which every closed interval $[a,
    b]$ is compact. Let $X$ be a closed subset of $E$ and $A \sub X$.
    Then either (i)~$\conv(A) \sub X$, or (ii)~$\conv(A)$ intersects $\rho X$.
\end{lemma}

Here, $\conv(A)$ denotes the convex hull in $E$ (not $X$). Generally, we
use $\conv$ to mean the convex hull in the largest space involved, which is
typically $\R^N$.

\begin{proof}
Suppose that $\conv(A) \cap \rho X = \emptyset$. By Remark~\ref{rmk:ic}, it
is enough to prove that $\intc^n(A) \sub X$ for all $n \geq 0$. For $n =
0$, this is trivial. Let $n \geq 1$. For all $x, y \in \intc^{n - 1}(A)$,
we have $[x, y] \sub \intc^n(A) \sub \conv(A)$; but $\conv(A) \cap \rho X =
\emptyset$, so $[x, y] \cap \rho X = \emptyset$. It follows from
Lemma~\ref{lem:boundary_in_interval} that $[x, y] \sub X$ for all $x, y \in
\intc^{n - 1}(A)$. Thus, $\intc^n(A) \sub X$, as required.
\end{proof}

The third and final result on convex hulls and inner boundaries is specific
to Euclidean space, and describes a dichotomy
(Figure~\ref{fig:boundary-crossing}(c)).

\begin{proposition}
\label{propn:dichotomy}
Let $X$ be a closed subset of $\R^N$, and let $x_0, \ldots, x_n \in X$ be
affinely independent points such that
\begin{equation}
\label{eq:dich-sub}
\conv\{x_0, \ldots, x_n\} \cap \rho X 
\sub
\{x_0, \ldots, x_n\}.
\end{equation}
Then either (i)~$\conv\{x_0, \ldots, x_n\} \sub X$, or (ii)~$\conv\{x_0,
\ldots, x_n\} \cap X = \{x_0, \ldots, x_n\}$.
\end{proposition}

Conditions~(i) and~(ii) are mutually exclusive when $n > 0$. The affine
independence hypothesis cannot be dropped: consider $X = [0, 1] \cup \{2\}
\sub \R$ and $(x_0, x_1, x_2) = (0, 1, 2)$.

\begin{proof}
Suppose that~(ii) does not hold. Then we can choose a point $x \in
\conv\{x_0, \ldots, x_n\} \cap X$ with $x \not\in \{x_0, \ldots, x_n\}$. We
must prove that $\conv\{x_0, \ldots, x_n\} \sub X$.

Define, for each $t \in [0, 1]$, 
\[
\Delta_t =
\conv \{(1 - t)x_0 + tx, \ldots, (1 - t)x_n + tx\}.
\]
First we claim that $\Delta_t \cap \rho X = \emptyset$ for all $t \in (0,
1]$.

Indeed, let $t \in (0, 1]$. Since $x \in \conv\{x_0, \ldots, x_n\}$, we
have $\Delta_t \sub \conv\{x_0, \ldots, x_n\}$ and so $\Delta_t \cap \rho X
\sub \{x_0, \ldots, x_n\}$ by~\eqref{eq:dich-sub}. Supposing for a
contradiction that $\Delta_t \cap \rho X \neq \emptyset$, it follows
without loss of generality that $x_0 \in \Delta_t$, so that $x_0$ is a
convex combination of $x_0, \ldots, x_n, x$ with a nonzero coefficient of
$x$. But $x$ is in the convex hull of $x_0, \ldots, x_n$, which are
affinely independent, forcing $x = x_0$ and contradicting our assumption
that $x \not\in \{x_0, \ldots, x_n\}$. This proves the claim that
$\Delta_t \cap \rho X = \emptyset$.

In particular, $(x_i, x] \cap \rho X = \emptyset$ for each $i \in \{0,
\ldots, n\}$. Since $x_i, x \in X$, Lemma~\ref{lem:boundary_in_interval}
then implies that $[x_i, x] \sub X$, giving $(1 - t)x_i + tx \in X$
for each $t \in (0, 1]$ and $i \in \{0, \ldots, n\}$. Hence $\Delta_t$ is
the convex hull of a subset of $X$, which allows us to apply
Lemma~\ref{lem:hull_contained} and deduce that $\Delta_t \sub X$, for each
$t \in (0, 1]$. And since $X$ is closed, it follows that $X$ contains
$\Delta_0 = \conv\{x_0, \ldots, x_n\}$.
\end{proof}

\section{Cores}
\label{sec:core}

Here we introduce the key convex-geometric player in our main theorem. Our
focus now is on subsets of Euclidean space $\R^N$. The notation $\conv$
and $\cconv$ will always refer to (closed) convex hulls in $\R^N$, and the
notation $[a, b]$ means the straight line segment from $a$ to $b$ in $\R^N$.

\begin{definition}
The \demph{core} of a subset $X \sub \R^N$ is $\core(X) = \cconv(\rho X)
\cap X$. 
\end{definition}

An example of a core is shown in Figure~\ref{fig:inner_boundary}(b). The
following further examples build on Examples~\ref{egs:boundary}. 

\begin{examples}
\label{egs:core}
\begin{enumerate}
\item 
The core of a closed annulus in $\R^2$ is the inner bounding circle.

\item
Let $X$ be $\R^2$ with an open disc and an open square removed, as in 
Figure~\ref{fig:equiv_spaces}(c). Then $\core(X)$ is as shown in
Figure~\ref{fig:equiv_spaces}(a). 

\item
The core of a finite subset $X$ of $\R^N$ with more than one point is $X$
itself.

\item
The core of the \textsf{L}-shaped space $X$ of Remark~\ref{rmk:L} is $X$.

\item
\label{eg:core-convex}
A closed subset of $\R^N$ has empty core if and only if it is convex. 
\end{enumerate}
\end{examples}

Inner boundaries are defined for abstract metric spaces, but cores are only
defined for metric spaces embedded in $\R^N$. Nevertheless, the notion of
core is intrinsic in the sense we now explain.

Write $\aff(X)$ for the affine hull of a set $X \sub \R^N$. The following
lemma is classical.

\begin{lemma}
\label{lemma:aff-ext}
Let $X \sub \R^N$, let $Y \sub \R^M$, and let $f: X \to Y$ be an invertible
isometry. Then $f$ extends uniquely to an invertible isometry $\aff(X) \to
\aff(Y)$. 
\end{lemma}

\begin{proof}
Existence is proved in Theorem~11.4 of Wells and
Williams~\cite{Wells1975}. We sketch the uniqueness argument. Let $\bar{f},
\tilde{f}: \aff(X) \to \aff(Y)$ be isometries extending $f$. Then for each
$a \in \aff(X)$, the set of points equidistant from $\bar{f}(a)$ and
$\tilde{f}(a)$ is affine and contains $Y$ (using the surjectivity of $f$),
and therefore contains $\aff(Y)$. In particular, it contains $\bar{f}(a)$,
so $\bar{f}(a) = \tilde{f}(a)$.
\end{proof}

We will consider the class of abstract metric spaces embeddable in $\R^N$.

\begin{definition}
\label{defn:euc-set}
A \demph{(closed) Euclidean set} is a metric space isometric to a (closed)
metric subspace of $\R^N$ for some $N \geq 0$.
\end{definition}

Let $X$ be a Euclidean set. Lemma~\ref{lemma:aff-ext} implies that
$\aff(X)$, $\conv(\rho X)$, $\cconv(\rho X)$, $\conv(\rho X) \cap X$ and
$\core(X)$ are all well-defined as metric spaces, up to isometry. For
example, if we embed $X$ isometrically into $\R^N$ in one way and into
$\R^M$ in another, then the core of the copy of $X$ in $\R^N$
is isometric to the core of the copy of $Y$ in $\R^M$. This is
the sense in which the core is an intrinsic construction. 

We also note that convexity is an intrinsic property of Euclidean sets $X$,
being equivalent to the property that for all $x, y \in X$, there exists an
isometry $[0, d(x, y)] \to X$ with $0 \mapsto x$ and $d(x, y) \mapsto y$.

Recall that the inner boundary construction is idempotent
(Remark~\ref{rmk:bdys}). We now show that the core construction is
idempotent too, and relate the two idempotents.

\begin{lemma}
\label{lemma:core-elem}
For a Euclidean set $X$,
\begin{enumerate}
\item 
\label{part:ce-sub}
$\rho X \sub \core(X)$;

\item
\label{part:ce-comp}
$\core(\rho X) = \rho X = \rho(\core(X))$;

\item
\label{part:ce-idem}
$\core(\core(X)) = \core(X)$.
\end{enumerate}
\end{lemma}

\begin{proof}
This result will not be needed and the proof is elementary, so we just
sketch it. Part~\bref{part:ce-sub} is immediate. The first identity
in~\bref{part:ce-comp} follows from $\rho$ being idempotent, and the second
is proved by a simple argument directly from the definitions, using the
convexity of $\cconv(\rho X)$. Part~\bref{part:ce-idem} follows from
the definitions and the second identity in~\bref{part:ce-comp}.
\end{proof}

Lemma~\ref{lem:restrict_inner_boundary} provides conditions under which
maps $\oppairu{X}{Y}$ of metric spaces restrict to mutually inverse
isometries between their inner boundaries. We now show that for
Euclidean sets, such maps also induce an isometry between the cores.

\begin{lemma}
\label{lemma:ib-to-core}
Let $\oppair{X}{Y}{f}{g}$ be maps between Euclidean sets. If $f$ and $g$
restrict to mutually inverse maps $\oppairu{\rho X}{\rho Y}$ then they also
restrict to mutually inverse maps $\oppairu{\core(X)}{\core(Y)}$. 
\end{lemma}

\begin{proof}
Suppose that $f$ and $g$ restrict to mutually inverse maps 
\[
\oppair{\rho X}{\rho Y}{f'}{g'}.
\]
It follows from Lemma~\ref{lemma:aff-ext} that $f'$ and $g'$ extend
uniquely to mutually inverse maps
\[
\oppair{\cconv(\rho X)}{\cconv(\rho Y)}{F}{G}.
\]
It suffices to show that $f(x) = F(x)$ for all $x \in \core(X)$, and
similarly for $g$. Evidently we need only prove the result for $f$.

Let $x \in \core(X)$. For each $y \in \rho Y$, the shortness of $f$ implies that 
\[
d(f(x), y) 
\leq
d(x, f'^{-1}(y))
=
d(F(x), F(f'^{-1}(y)))
=
d(F(x), y).
\]
Choosing an embedding of $Y$ into $\R^M$, the set
\[
V = \{ b \in \R^M : d(f(x), b) \leq d(F(x), b) \}
\]
therefore contains $\rho Y$. But $V$ is either $\R^M$ or a closed
half-space of it, so it is closed and convex; hence $\cconv(\rho Y) \sub
V$. In particular, $F(x) \in V$, giving $f(x) = F(x)$.
\end{proof}

\begin{remark}
    Note that $f$ and $g$ in the statement of the previous lemma need not be isometries, even though their respective restrictions are. For example, the retraction $X \to \core(X)$ constructed in the next proposition is not an isometry and yet, together with the inclusion $\core(X) \to X$, gives a pair of maps satisfying the hypotheses of Lemma~\ref{lemma:ib-to-core}.
\end{remark}

For a nonempty convex closed subset $C$ of $\R^N$, every point $x \in \R^N$
has a unique closest point $\pi(x) \in C$. This defines a
short retraction $\pi: \R^N \to C$ of the inclusion $C
\hookrightarrow \R^N$, called the \demph{metric projection} onto $C$
(Theorem~1.2.1 of~\cite{Schneider2013}).

For a \emph{non}convex closed subset $X$ of $\R^N$, the inner boundary
is nonempty (Example~\ref{egs:boundary}\bref{eg:bdy-menger}), so we have a
metric projection map $\R^N \to \cconv(\rho X)$.

\begin{proposition}
\label{propn:core-retract}
Let $X$ be a \emph{non}convex closed subset of $\R^N$. Then metric projection
$\R^N \to \cconv(\rho X)$ restricts to a map $X \to \core(X)$. 
\end{proposition}

Thus, every point of $X$ has a unique closest point in $\core(X)$.

\begin{proof}
Write $\pi: \R^N \to \cconv(\rho X)$ for metric projection. We have to show
that whenever $x \in X$, then also $\pi(x) \in X$. Certainly $[x, \pi(x)]
\sub \cconv(X)$, so by Proposition~\ref{propn:conv-union}, 
\[
[x, \pi(x)]
=
\bigl( [x, \pi(x)] \cap X \bigr) 
\cup
\bigl( [x, \pi(x)] \cap \cconv(\rho X) \bigr).
\]
Both sets in this union are closed in $\R^N$, and nonempty since $x$ belongs
to the first and $\pi(x)$ to the second. Since $[x, \pi(x)]$ is connected,
$[x, \pi(x)] \cap X \cap \cconv(\rho X)$ contains some point $y$. Then $y
\in \cconv(\rho X)$ with $d(x, y) \leq d(x, \pi(x))$, which by definition
of $\pi$ implies that $y = \pi(x)$. Hence $\pi(x) \in X$.
\end{proof}

The final result of this section will not be needed for the main theorem,
but is a basic property of cores.

\begin{proposition}
\label{prop:core-closure}
Let $X$ be a closed Euclidean set. Then $\core(X) = \overline{\conv(\rho X)
\cap X}$. 
\end{proposition}

\begin{proof}
The right-to-left inclusion is clear. For the converse, let $x \in
\core(X)$ and $\epsilon > 0$. We must prove that
\begin{equation}
\label{eq:cc-neq}
B_\epsilon(x) \cap \conv(\rho X) \cap X \neq \emptyset,
\end{equation}
where $B_\epsilon(x)$ denotes the open ball in $\R^N$. This is immediate if
$B_\epsilon(x) \cap \rho X \neq \emptyset$, so assume that $B_\epsilon(x)
\cap \rho X = \emptyset$.

We will use two properties of $B_\epsilon(x) \cap X$. First, it is
convex. For let $p, q \in B_\epsilon(x) \cap X$. Then $[p, q] \sub
B_\epsilon(x)$, so $[p, q] \cap \rho X = \emptyset$. By 
Lemma~\ref{lem:boundary_in_interval} or Lemma~\ref{lem:hull_contained}, it
follows that $[p, q] \sub X$, proving convexity.

The second property is that if $p \in B_\epsilon(x) \cap X$ and $y \in \rho
X$, then $[p, y] \cap B_\epsilon(x) \sub X$. For let $q \in [p, y] \cap
B_\epsilon(x)$. Then $[p, q] \sub B_\epsilon(x)$, so $[p, q) \cap \rho X =
\emptyset$, and then Lemma~\ref{lem:boundary_in_interval} applied to $q \in
[p, y]$ gives $q \in X$.

Since $x \in \cconv(\rho X)$, the set $\overline{B_{\epsilon/2}(x)} \cap
\conv(\rho X)$ is nonempty. We prove that it is a subset of $X$, which
will imply property~\eqref{eq:cc-neq} and complete the proof.

Thus, let $a \in \overline{B_{\epsilon/2}(x)} \cap \conv(\rho X)$
(Figure~\ref{fig:closure}). We must prove that $a \in X$. Choose a point $z
\in \overline{B_{\epsilon/2}(x)} \cap X$ minimising $d(a, z)$. It is enough
to prove that $z = a$.

\begin{figure}
    \centering
    \def\svgwidth{0.7\columnwidth}
    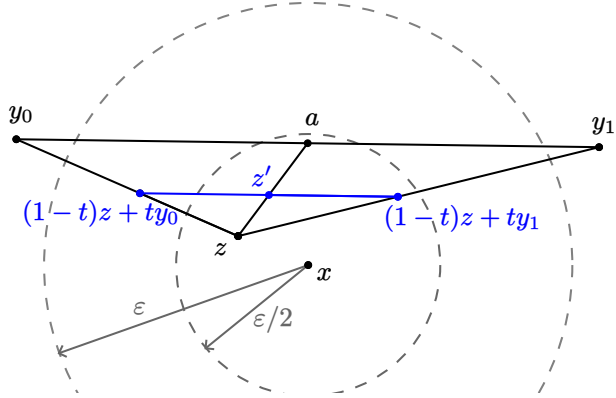

    \caption{The proof of Proposition~\ref{prop:core-closure}.}
    \label{fig:closure}
\end{figure}

Since $a \in \conv(\rho X)$, we have $a = \sum_{i = 0}^n p_i y_i$ for some
$n \geq 0$, points $y_i \in \rho X$, and nonnegative $p_i$ with $\sum
p_i = 1$. Since $z \in B_\epsilon(x)$, we can choose $t \in (0, 1]$ such
that $(1 - t)z + ty_i \in B_\epsilon(x)$ for all $i \in \{0, \ldots, n\}$.

The second property above implies that $[z, y_i] \cap B_\epsilon(x) \sub
X$, and in particular, $(1 - t)z + ty_i \in X$, for each $i$. Hence by the
first property, 
\[
\sum_{i = 0}^n p_i ((1 - t)z + ty_i) \in X,
\]
that is, the point $z' = (1 - t)z + ta$ is in $X$. Moreover, $z, a \in
\overline{B_{\epsilon/2}(x)}$, so $z' \in \overline{B_{\epsilon/2}(x)} \cap
X$. But $z$ was defined to be the point of $\overline{B_{\epsilon/2}(x)} \cap
X$ minimising the distance to $a$, and $d(a, z') = (1 - t) d(a, z)$ with $t
> 0$, so $z = a$.
\end{proof}

\section{Magnitude homology of metric spaces}
\label{sec:mhms}

Here we review the definition of the magnitude homology of a metric space,
first introduced in~\cite{Leinster2021} as a special case of the magnitude
homology of an enriched category.

Let $X = (X, d)$ be a metric space. Write $\R^+$ for the set of nonnegative
real numbers.

\begin{definition}\label{def:mag-ch-cx}
    Let $n \geq 0$. A \demph{proper chain in $X$ of degree $n$} is an $(n +
    1)$-tuple $\bm{x} = (x_0, \ldots,x_n)$ of points in $X$ such that $x_{i
    - 1} \neq x_i$ whenever $1 \leq i \leq n$. Its \demph{length} is
    $d(x_0,x_1) + \cdots + d(x_{n-1},x_n)$. Write $\Pc{n}(X)$ for the set
    of proper chains in $X$ of degree $n$, and for $\ell \in \R^+$, write
    $\PC{n}{\ell}(X) = \{\bm{x} \in \Pc{n}(X) : \bm{x} \text{ has length }
    \ell\}$.
\end{definition}

Let $\Z S$ denote the free abelian group on a set $S$.

\begin{definition}
    For $n \geq 0$ and $\ell \in \R^+$, put $\MC{n}{\ell}(X) = \Z
    \PC{n}{\ell}(X)$. 
    The \demph{magnitude chain complex} of $X$ is an $\R^+$-graded chain complex
    \[\MC{*}{*}(X) = \bigoplus_{\ell \in \R^+} \MC{*}{\ell}(X).\]
    The boundary map
    $\partial_n : \MC{n}{\ell}(X) \to \MC{n-1}{\ell}(X)$ is $\sum_{i=0}^n
    (-1)^i \partial_n^i$, where $\partial_n^i$ is defined on generators by
    \[\partial_n^i(x_0,\ldots,x_n) = \begin{cases}
        (x_0,\ldots,\widehat{x_i},\ldots,x_n) & \text{if }
        (x_0,\ldots,\widehat{x_i},\ldots,x_n) \text{ has length } \ell, \\ 
        0 &\text{otherwise.}
    \end{cases}\]
    Here $(x_0,\ldots,\widehat{x_i},\ldots,x_n)$ denotes the tuple resulting
    from removing the $i$th entry of $(x_0,\ldots,x_n)$. The
    \demph{magnitude homology} $\MH{*}{*}(X)$ of $X$ is the homology of
    $\MC{*}{*}(X)$.
\end{definition}

Recall that maps of metric spaces are by definition short
(Definition~\ref{defn:map}). Any such map $f: X \to Y$ induces a chain map
$f_\# : \MC{*}{*}(X) \to \MC{*}{*}(Y)$, given on the generating set
$\PC{n}{\ell}(X)$ by
\[f_\#(x_0,\ldots,x_n) = \begin{cases}
 (f(x_0), \ldots, f(x_n)) &\text{if } (f(x_0), \ldots, f(x_n)) \in \PC{n}{\ell}(Y), \\
 0 &\text{otherwise.}
\end{cases}\]
In turn, $f_\#$ induces a map $f_* : \MH{*}{*}(X) \to \MH{*}{*}(Y)$ in
homology. Thus, $\MC{*}{*}$ is a functor from the category $\Met$ of metric
spaces and their maps to the category of $\R^+$-graded chain complexes of
abelian groups, and $\MH{*}{*}$ is a functor from $\Met$ to $(\N \times
\R^+)$-graded abelian groups.

\begin{remark}\label{rmk:0_homology}
    Some trivial cases are easily described. Clearly $\MH{n}{0}(X) = 0$
    for all $n \geq 1$. Also, the boundary map $\partial_1 :
    \MC{1}{\ell}(X) \to \MC{0}{\ell}(X)$ is zero for any $\ell \in \R^+$,
    so 
    \[\MH{0}{\ell}(X) = \begin{cases}
        \Z X &\text{if } \ell = 0,\\
        0 &\text{otherwise.}
    \end{cases}\]
\end{remark}

With only a little more effort, one shows that $\MH{1}{\ell}(X)$ is the
free abelian group on the set of ordered pairs of adjacent points distance
$\ell$ apart (Corollary~4.5 of~\cite{Leinster2021}). In particular, a
closed subset of $\R^N$ has trivial first magnitude homology if and only if
it is convex.

The higher magnitude homology groups are more subtle. Even in the case of
graphs, seen as metric spaces as in Section~\ref{sec:aligned}, the
magnitude homology groups can have torsion (Corollary~5.12(3) of Kaneta and
Yoshinaga~\cite{Kaneta2021}). Going further, Sazdanovic and Summers showed
that every finitely generated abelian group arises as a subgroup of some
magnitude homology group of some graph (Theorem~3.14 of~\cite{SaSu}).

\section{Magnitude homology of aligned spaces}
\label{sec:mhst}

When a space is aligned, its higher magnitude homology groups have a simple
description similar to that of $\MH{1}{*}$ above. More exactly, Kaneta and
Yoshinaga showed that all the magnitude homology groups of an aligned space
are free, and they identified a basis, as follows.

\begin{definition}
Let $X$ be a metric space, $n \geq 0$ and $\ell \in \R^+$.  A \demph{thin
chain} of degree $n$ and length $\ell$ is a proper chain $\bm{x} \in
\PC{n}{\ell}(X)$ such that $x_{i - 1}$ is adjacent to $x_i$ for all $i \in
\{1, \ldots, n\}$ and $x_i \not\in [x_{i - 1}, x_{i + 1}]$ for all $i \in
\{1, \ldots, n - 1\}$. Write 
    \[\THIN{n}{\ell}(X) = \{ \bm{x} \in \PC{n}{\ell}(X) : \bm{x} \text{ is
    thin}\}.\] 
\end{definition}

Any thin chain is a cycle, so there is a map of sets $\THIN{n}{\ell}(X) \to
\MH{n}{\ell}(X)$ assigning to each thin chain its homology class. This map
extends uniquely to a homomorphism $\Z \THIN{n}{\ell}(X) \to
\MH{n}{\ell}(X)$.

Kaneta and Yoshinaga proved that when $X$ is aligned, this canonical
homomorphism is an isomorphism (Theorem~5.2 of~\cite{Kaneta2021}). Thus,
the magnitude homology groups of an aligned space are freely generated by
the thin chains. 

In fact, they proved this under the more careful hypotheses
that $X$ is geodetic and has no 4-cuts at certain length scales. By
Proposition~\ref{prop:aligned-eqv}, the cruder assumption of alignedness
suffices. Although some of our results hold under more careful hypotheses
too, we assume alignedness throughout in order to simplify the exposition.

The isomorphism $\Z \THIN{n}{\ell} \cong \MH{n}{\ell}$ is natural in the
following sense. A map $f: X \to Y$ induces a homomorphism $f_\star: \Z
\THIN{n}{\ell}(X) \to \Z \THIN{n}{\ell}(Y)$, defined on generators by
\[
f_\star(x_0, \ldots, x_n) 
=
\begin{cases}
(f(x_0), \ldots, f(x_n))        &
\text{if } (f(x_0), \ldots, f(x_n)) \in \THIN{n}{\ell}(Y),      \\
0       &
\text{otherwise}.
\end{cases}
\]
In this way, $\Z\THIN{n}{\ell}$ becomes a functor from aligned spaces to
abelian groups.

\begin{theorem}[Kaneta and Yoshinaga]
\label{thm:thin_chains}
Let $n \geq 0$ and $\ell \in \R^+$. For aligned spaces $X$, the canonical
homomorphism $\Z\THIN{n}{\ell}(X) \to \MH{n}{\ell}(X)$ is an isomorphism,
natural in $X$.
\end{theorem}

\begin{proof}
The main statement follows from Theorem~5.2 of~\cite{Kaneta2021}, using
Proposition~\ref{prop:aligned-eqv} above. The naturality is not stated
explicitly in~\cite{Kaneta2021}, but is readily checked.
\end{proof}

For $n \in \{0, 1\}$, Theorem~\ref{thm:thin_chains} reproduces
the descriptions of $\MH{0}{*}$ and $\MH{1}{*}$ at the end of
Section~\ref{sec:mhms}, which do not require alignedness.

For any thin chain $(x_0, \ldots, x_n)$ with $n \geq 1$, the points $x_i$
are all in the inner boundary $\rho X$. Hence the natural isomorphism of
Theorem~\ref{thm:thin_chains} gives:

\begin{corollary}
\label{cor:parallel-ib}
Let $f, g: X \to Y$ be maps of aligned spaces. Let $n \geq 1$ and $\ell
\in \R^+$. If $f(x) = g(x)$ for all $x \in \rho X$ then $f_* = g_*:
\MH{n}{\ell}(X) \to \MH{n}{\ell}(Y)$.
\end{corollary}

A thin chain in degree $0$ is just a point, not necessarily in the inner
boundary. By Remark~\ref{rmk:0_homology}, the maps $f_*, g_*: \MH{0}{*}(X)
\to \MH{0}{*}(Y)$ are only equal when $f = g$.

Corollary~\ref{cor:parallel-ib} suggests an analogy: perhaps the condition
that $f|_{\rho X} = g|_{\rho X}$ plays a similar role for magnitude
homology of metric spaces as homotopy between maps plays for ordinary
homology of topological spaces. We digress briefly (until
Remark~\ref{rmk:ib-2cat}) to develop this idea, proving a stronger version
of Corollary~\ref{cor:parallel-ib}.

\begin{theorem}\label{thm:homotopic}
    Let $f, g : X \to Y$ be maps of metric spaces, with $X$ aligned. Let
    $\ell \in \R^+$. If $f(x) = g(x)$ for all $x \in \rho X$, then $f_\#,
    g_\# : \MC{*}{\ell}(X) \to \MC{*}{\ell}(Y)$ are chain homotopic in
    positive degree.
\end{theorem}

To prove Theorem~\ref{thm:homotopic}, we use Kaneta and Yoshinaga's
notion of frame (Definition 3.3 of~\cite{Kaneta2021}). 

\begin{definition}
    For $0 \leq i \leq n$, a proper chain $\bm{x} = (x_0,\ldots,x_n) \in
    \Pc{n}(X)$ is \demph{smooth} at $i$ if $0 < i < n$ and $x_{i-1} \prec
    x_i \prec x_{i+1}$; otherwise, it is \demph{singular} at $i$. Writing
    $0 = i_0 < i_1 < \cdots < i_k = n$ for the indices at which $\bm{x}$ is
    singular, the \demph{frame} of $\bm{x}$ is $(x_{i_0}, x_{i_1}, \ldots,
    x_{i_k})$.
\end{definition}

Recall that $\partial_n^i(\bm{x}) = (x_0, \ldots, \widehat{x_i}, \ldots,
x_n)$ when $\bm{x}$ is smooth at $i$.

\begin{lemma}
\label{lemma:same-frame}
Let $X$ be an aligned space, let $\bm{x} = (x_0, \ldots, x_n) \in
\Pc{n}(X)$, and let $i \in \{0, \ldots, n\}$. If $\bm{x}$ is smooth at $i$
then $\partial_n^i(\bm{x})$ and $\bm{x}$ have the same frame.
\end{lemma}

\begin{proof}
This is proved by an elementary argument similar to the proof of
Proposition~3.7 of Kaneta and Yoshinaga~\cite{Kaneta2021}.
\end{proof}

\begin{proof}[Proof of Theorem~\ref{thm:homotopic}]
    We will define $\phi : \MC{n}{\ell}(X) \to
    \MC{n+1}{\ell}(Y)$ for each $n \geq 0$ in such a way that
    \[
    g_\# - f_\# = \partial \phi + \phi \partial: 
    \MC{n}{\ell}(X) \to \MC{n}{\ell}(Y) 
    \]
    for each $n \geq 1$. 

    To this end, for each pair of distinct non-adjacent points $x,x' \in
    X$, choose a point $xx'$ in the interval $(x,x')$. Let $n \geq 0$ and
    let $\bm{x} = (x_0,\ldots,x_n)$ be a proper chain, with frame $(x_0 =
    x_{i_0}, x_{i_1}, \ldots, x_{i_k} = x_n)$. If there is some $r \in
    \{1,\ldots,k\}$ such that $x_{i_{r-1}}$ is not adjacent to $x_{i_r}$,
    take the smallest such $r$. Then, by alignedness, there is a unique
    index $i_{r-1} < h \leq i_r$ such that $x_{i_{r-1}}x_{i_r} \in
    (x_{h-1},x_h)$, and we put
    \[\bm{x}' = (x_0,\ldots,x_{h-1}, x_{i_{r-1}}x_{i_r},x_h,\ldots,x_n).\]
    Otherwise, put $\bm{x}' = 0$. Finally, set $\phi(\bm{x}) = (-1)^h(g_\#
    - f_\#)(\bm{x}')$, which defines $\phi$ on the generators of
    $\MC{n}{\ell}(X)$.

    By Lemma~\ref{lemma:same-frame}, all the nonzero terms
    $\partial_n^j(\bm{x})$ in the sum $\partial(\bm{x}) = \sum (-1)^j
    \partial_n^j(\bm{x})$ have the same frame as $\bm{x}$. Using this, one
    finds that if the condition for $\bm{x}'$ to be nonzero holds, then the
    same is true for $\partial^j_n(\bm{x})$ whenever $\partial^j_n(\bm{x})$
    is nonzero. In any case, 
    \[(\partial^j_n(\bm{x}))' = \begin{cases}
        \partial^j_{n+1}(\bm{x}') & \text{if } j < h, \\
        \partial^{j+1}_{n+1}(\bm{x}') & \text{otherwise},
        \end{cases}\]
    where we make the convention that $0' = 0$. In turn, this implies that
    \begin{equation}\label{eq:phi_bdd}
    \phi(\partial^j_n(\bm{x})) = \begin{cases}
        - \partial^j_{n+1} \phi(\bm{x}) & \text{if } j < h, \\
        \partial^{j+1}_{n+1}\phi(\bm{x}) & \text{otherwise.}
    \end{cases}
    \end{equation}

    To prove that $\phi$ is a chain homotopy, it suffices to show that
    \begin{equation}\label{eq:chain_homotopy}
        (g_\# - f_\#)(\bm{x}) = (\partial \phi + \phi \partial)(\bm{x}).
    \end{equation}
    If $\bm{x}' = 0$, then each component in the frame of $\bm{x}$ is adjacent to the next one, so $\bm{x}$ must equal its frame, and all the components of $\bm{x}$ are in $\rho X$. In this case, it is clear that both sides of \eqref{eq:chain_homotopy} are zero. The case where $\bm{x}'$ is nonzero is a routine calculation using \eqref{eq:phi_bdd} and the fact that $\partial^h_{n+1}(\bm{x}') = \bm{x}$.
\end{proof}

\begin{remark}
As an alternative proof, the existence of such a chain homotopy also
follows from Theorem~\ref{thm:thin_chains} and some general homological
algebra. Let $C$ be a chain complex of free abelian groups with free
homology, and $Z_n = \ker \partial_n \subseteq C_n$, $B_n = \im
\partial_{n+1} \subseteq C_n$ and $H_n = Z_n / B_n$ as usual. The $H_n$
assemble into a complex $H$ with trivial differentials. If $s_n : H_n \to
Z_n$ is a section of the quotient map $q_n : Z_n \to H_n$ for each $n$,
then the composites
\setlength\mathsurround{0pt}
\begin{equation}\label{eq:sec_equiv}
    \begin{tikzcd}
    H_n \ar[r, "s_n"] & Z_n \ar[r, hook] & C_n
    \end{tikzcd}
\end{equation}
\setlength\mathsurround{0.8pt}
form a chain homotopy equivalence $H \simeq C$.

To see this, note that since a subgroup of a free abelian group is free,
all of $C_n$, $Z_n$, $B_n$ and $H_n$ are free. In particular, $B_{n - 1}$
is free, so the short exact sequence
\setlength\mathsurround{0pt}
\[\begin{tikzcd}
0 \rar & Z_n \rar & C_n \rar{\partial} & B_{n-1} \rar & 0
\end{tikzcd}\]
\setlength\mathsurround{0.8pt}
splits, giving $C_n \cong Z_n \oplus B_{n-1}$ for each $n$. It follows that
$C$ is the direct sum of complexes of the form
\setlength\mathsurround{0pt}
\[\begin{tikzcd}
0 \ar[r] & B_n \ar[r, hook] & Z_n \rar & 0.
\end{tikzcd}\]
\setlength\mathsurround{0.8pt}
Truncating this complex at $Z_n$ gives a free resolution of $H_n$. Since
$H_n$ is free, it is also a free resolution of itself. The maps $q_n$ and
$s_n$ give lifts of the identity on $H_n$ to chain maps between these two
resolutions, and are then inverse homotopy equivalences by the uniqueness
up to chain homotopy of such lifts. Composing the maps $s_n$ with the
inclusions $Z_n \hookrightarrow C_n$ and then taking the direct sum over
all $n$ gives the desired chain homotopy equivalence $H \simeq C$.

For an aligned space $X$, Theorem~\ref{thm:thin_chains} implies that
$H_{*,\ell}(X)$ is free. We may take $s_n$ to be the map that sends each
thin chain to itself, giving a chain homotopy equivalence $\psi :
H_{*,\ell}(X) \to C_{*,\ell}(X)$ as in~\eqref{eq:sec_equiv}, with homotopy
inverse $\phi$. Now if $f,g : X \to Y$ are equal on $\rho X$ then $f_\#
\psi = g_\# \psi$ in positive degree, and so
\[f_\# \simeq f_\# \psi \phi = g_\# \psi \phi \simeq g_\#.\]
\end{remark}

\begin{remark}
\label{rmk:ib-2cat}
Here we return to the analogy between, on the one hand, the condition that
two maps $\parpairu{X}{Y}$ of metric spaces agree on the inner boundary
of $X$, and, on the other, the condition that two maps $\parpairu{X}{Y}$ of
topological spaces are homotopic. The analogy is only loose, since
equality on the inner boundary is not stable under 2-categorical
composition, unlike topological homotopy. Indeed, there are examples of
maps of metric spaces
\setlength\mathsurround{0pt}
\[\begin{tikzcd}
X \ar[r, "f"] & Y \ar[r, shift left, "g"] \ar[r, shift right, "g'"'] & Z 
\end{tikzcd}\]
\setlength\mathsurround{0.8pt}
such that $g|_{\rho Y} = g'|_{\rho Y}$ but $(gf)|_{\rho X} \neq
(g'f)|_{\rho X}$. (Take $f$ to be the inclusion $\{0, 1\} \hookrightarrow
[0, 1]$ and $g, g': [0, 1] \to \R$ to be any two maps that differ at $0$ or
$1$.) However, by Theorem~\ref{thm:homotopic}, it is still the case that
$(gf)_\#$ and $(g'f)_\#$ are chain homotopic in positive degree.

It is an open problem to find a compact description of the equivalence
relation on maps of metric spaces that is generated by equality on the
inner boundary and closed under 2-categorical composition. Two maps that
are equivalent in this sense are guaranteed by Theorem~\ref{thm:homotopic}
to induce chain homotopic chain maps in positive degree.

Different connections between homotopy and magnitude homology have been
investigated by Tajima and Yoshinaga~\cite{TaYo}.
\end{remark}

This ends the digression on homotopy, and we return to the question of when
two maps between aligned spaces induce the same map on homology in positive
degree. From here on, we often suppress the length index $\ell$, writing
$\Mh{n}(X)$ for the $\R^+$-graded abelian group $\bigoplus_{\ell \in \R^+}
\MH{n}{\ell}(X)$.

\begin{proposition}\label{prop:positive_tfae}
    Let $e : X \to X$ be an endomorphism of an aligned metric space. The following are equivalent:
    \begin{enumerate}
        \item \label{part:ptfae-all}
        $e_*: \Mh{n}(X) \to \Mh{n}(X)$ is the identity for all $n
        \geq 1$; 
        \item \label{part:ptfae-some}
        $e_*: \Mh{n}(X) \to \Mh{n}(X)$ is the identity for some $n
        \geq 1$; 
        \item \label{part:ptfae-ib}
        $e(x) = x$ for all $x \in \rho X$.
    \end{enumerate}
\end{proposition}

\begin{proof}
    That~\bref{part:ptfae-all} implies~\bref{part:ptfae-some} is
    immediate. Now assume~\bref{part:ptfae-some}, and choose such an
    $n$. To prove~\bref{part:ptfae-ib}, let $x \in \rho X$, and choose $x'
    \in X$ adjacent to $x$.  Then the alternating $(n + 1)$-tuple
    $(x,x',x,x',\ldots)$ is a thin chain. By the naturality of the
    isomorphism in Theorem~\ref{thm:thin_chains} 
    and alignedness, $e_\star(\bm{x}) = \bm{x}$, giving $(x, x', \ldots) =
    (e(x), e(x'), \ldots)$ and so $e(x) = x$.

    Lastly, assuming~\bref{part:ptfae-ib}, we
    prove~\bref{part:ptfae-all}. Let $n \geq 1$. By naturality and
    alignedness again, it is enough to prove that $e_\star(\bm{x}) =
    \bm{x}$ for each $\bm{x} = (x_0, \ldots, x_n) \in \THIN{n}{*}(X)$. But
    $x_0, \ldots, x_n \in \rho X$ since $n \geq 1$, so $e(x_i) = x_i$ for
    each $i$, as required.
\end{proof}

We now show that for two maps $\oppairu{X}{Y}$ of metric spaces to be
mutually inverse in positive degree magnitude homology is equivalent to a
concrete geometric condition.

\begin{theorem}
\label{thm:op-tfae}
Let $\oppair{X}{Y}{f}{g}$ be maps of aligned metric spaces. The following
are equivalent:
\begin{enumerate}
\item 
\label{part:ot-all}
the maps $\oppair{\Mh{n}(X)}{\Mh{n}(Y)}{f_*}{g_*}$ are mutually
inverse for all $n \geq 1$;

\item
\label{part:ot-some}
the maps $\oppair{\Mh{n}(X)}{\Mh{n}(Y)}{f_*}{g_*}$ are mutually
inverse for some $n \geq 1$;

\item
\label{part:ot-ib}
$f$ and $g$ restrict to mutually inverse isometries $\oppairu{\rho X}{\rho
Y}$.
\end{enumerate}
\end{theorem}

\begin{proof}
\bref{part:ot-all}$\implies$\bref{part:ot-some} is trivial. 

For \bref{part:ot-some}$\implies$\bref{part:ot-ib}, take $n \geq 1$ as
in~\bref{part:ot-some}. Then $(gf)_*: \Mh{n}(X) \to \Mh{n}(X)$ is the
identity, so by Proposition~\ref{prop:positive_tfae}, $gf(x) = x$ for all
$x \in \rho X$. Similarly, $fg(y) = y$ for all $y \in \rho
X$. Then~\bref{part:ot-ib} follows from
Lemma~\ref{lem:restrict_inner_boundary}\bref{part:rib-fg}. 

\bref{part:ot-ib}$\implies$\bref{part:ot-all} follows from
Proposition~\ref{prop:positive_tfae} applied to $gf$ and $fg$. 
\end{proof}

At this point, it follows easily that the positive-degree magnitude homology of a \emph{non}convex closed Euclidean set is isomorphic to that of its core. Indeed, the retraction $X \to \core(X)$ of Proposition~\ref{propn:core-retract} and the inclusion $\core(X) \to X$ satisfy condition \bref{part:ot-ib} of the theorem by Lemma~\ref{lemma:core-elem}\bref{part:ce-sub}. We will return to this fact in the next section (see Theorem~\ref{thm:equiv-to-core}).

\begin{remarks}
\label{rmks:opposing-maps}
The following counterexamples illustrate the essential role of the two
opposing maps in Theorem~\ref{thm:op-tfae}.
\begin{enumerate}
\item 
\label{rmk:om-ib}
Spaces with isometric inner boundaries need not have isomorphic magnitude
homology groups in positive degree.  For example, consider $X = \{0, 1, 2,
3\}$ and $Y = \{0\} \cup [1,2] \cup \{3\}$, metrised as subspaces of
$\R$. Then $\rho X = \rho Y = X$. However,
\begin{align*}
\THIN{1}{1}(X)  &
=
\{ (0, 1), (1, 0), (1, 2), (2, 1), (2, 3), (3, 2) \},   \\
\THIN{1}{1}(Y)  &
=
\{ (0, 1), (1, 0), (2, 3), (3, 2) \},   
\end{align*}
giving $\MH{1}{1}(X) = 6\Z$ but $\MH{1}{1}(Y) = 4\Z$ by
Theorem~\ref{thm:thin_chains}.

\item
There are examples of maps of aligned spaces $f: X \to Y$ such that $f_*:
\Mh{n}(X) \to \Mh{n}(Y)$ is an isomorphism for some but not all $n
\geq 1$.

Indeed, let $X = \{0,1,2\} \subseteq \R$. Let $v_0, v_1, v_2 \in \R^2$ be
the vertices of an equilateral triangle of edge length $1$, and let $Y$ be
$\conv\{v_0, v_1, v_2\}$ with the open line segments $(v_0, v_1)$ and
$(v_1, v_2)$ removed. Define $f : X \to Y$ by $f(i) = v_i$ for $i = 0, 1,
2$. We have
\begin{align*}
\THIN{1}{1}(X)  &
= \{(0, 1), (1, 0), (1, 2), (2, 1)\},   \\
\THIN{1}{1}(Y)  &
= \{(v_0, v_1), (v_1, v_0), (v_1, v_2), (v_2, v_1)\},   
\end{align*}
with $\THIN{1}{\ell}(X) = \emptyset = \THIN{1}{\ell}(Y)$ for $\ell \neq
1$. Hence by Theorem~\ref{thm:thin_chains}, $f_*: \MH{1}{*}(X) \to
\MH{1}{*}(Y)$ is an isomorphism. On the other hand,
\begin{align*}
\THIN{2}{1}(X)  &
= \{(0, 1, 0), (1, 0, 1), (1, 2, 1), (2, 1, 2)\},   \\
\THIN{2}{1}(Y)  &
= \{(v_0, v_1, v_0), (v_0, v_1, v_2), (v_1, v_0, v_1), (v_1, v_2, v_1),
(v_2, v_1, v_0), (v_2, v_1, v_2)\}.   
\end{align*}
Hence $\MH{2}{1}(X) = 4\Z$ and $\MH{2}{1}(Y) = 6\Z$, and so $f_*:
\MH{2}{*}(X) \to \MH{2}{*}(Y)$ is not an isomorphism.

\item 
\label{rmk:om-circles}
There are also examples of aligned spaces $X$ and $Y$ such that
$\MH{n}{\ell}(X) \cong \MH{n}{\ell}(Y)$ for all $n \geq 0$ and $\ell \in
\R^+$, but for which there is no map $X \to Y$ inducing isomorphisms in
homology.

Let $X \subseteq \R^2$ be the unit circle centred at the origin, with the
subspace metric. Let $Y \subseteq \R^2$ be the union of $X$ and a circle of
radius $1$ centred at $(2,0)$. In both spaces, the set of pairs of adjacent
points distance $\ell$ apart has continuum cardinality when $\ell \leq 2$
and is empty otherwise. By considering chains of the form $(z, z', z, z',
\ldots)$, we deduce that $\THIN{n}{\ell}(X)$ and $\THIN{n}{\ell}(Y)$ have
continuum cardinality for all $\ell \leq 2n$ and are empty otherwise. It
follows from Theorem~\ref{thm:thin_chains} that $\MH{*}{*}(X) \cong
\MH{*}{*}(Y)$.

We now show that there is no map $f: X \to Y$ such that $f_*: \MH{1}{*}(X)
\to \MH{1}{*}(Y)$ is an isomorphism of $\R^+$-graded abelian
groups. Suppose for a contradiction that such a map $f$ exists.

Theorem~\ref{thm:thin_chains} implies that $\MH{1}{*}(X)$ is the free
abelian group on the set of ordered pairs of adjacent points in $X$, and
similarly for $Y$. It also implies that $f_\star: \Z \THIN{1}{\ell}(X) \to
\Z \THIN{1}{\ell}(Y)$ is an isomorphism for each $\ell \in \R^+$. In
particular, $f_\star$ is surjective, which since $\rho Y = Y$ implies that
$f$ is surjective. This is a contradiction: there is no short
surjection $X \to Y$, since the diameter of $Y$ is strictly larger than
that of $X$.
\end{enumerate}
\end{remarks}

\section{Magnitude homology equivalence}
\label{sec:mheq}

Theorem~\ref{thm:op-tfae} tells us when two maps of aligned metric spaces
$\oppairu{X}{Y}$ induce mutually inverse isomorphisms in magnitude
homology. This raises a question: given only the spaces $X$ and $Y$, when
does such a pair of maps exist? This is the question answered by the main
theorem, in the next section. Here we build up to it by considering
inclusions and retractions. In keeping with Definition~\ref{defn:map}, a
\demph{retraction} is by definition short, and the term
\demph{retract} is used accordingly.

\begin{proposition}
\label{propn:incl}
Let $X$ be an aligned metric space and $C$ a convex subset of $X$ such that
$\rho X \sub C$. Then the inclusion $C \hookrightarrow X$ induces an
isomorphism $\Mh{n}(C) \to \Mh{n}(X)$ for all $n \geq 1$.
\end{proposition}

\begin{proof}
Write $\iota: C \hookrightarrow X$ for the inclusion. By
Theorem~\ref{thm:thin_chains}, it is enough
to show that for each $n \geq 1$ and $\ell \in \R^+$, the map 
\[
\iota_\star: \Z \THIN{n}{\ell}(C) \to \Z \THIN{n}{\ell}(X)
\]
is an isomorphism.

It is injective if every thin chain in $C$ is a thin chain in $X$, which is
true by the convexity of $C$ in $X$. It is surjective if every thin chain
$(x_0, \ldots, x_n)$ in $X$ is a thin chain in $C$. Since $n \geq 1$, we
have $x_i \in \rho X \sub C$ for each $i \in \{0, \ldots, n\}$. Hence
$(x_0, \ldots, x_n)$ is a chain in $C$, and thin in $C$ since it is thin in
$X$.
\end{proof}

Taking $C = \emptyset$ in Proposition~\ref{propn:incl} gives the following
corollary, which also follows directly from Theorem~\ref{thm:thin_chains}
(as Kaneta and Yoshinaga observed in the introduction
to~\cite{Kaneta2021}), and was proved independently by Jubin as Theorem~7.2
of~\cite{Jubi}.

\begin{corollary}\label{cor:empty_inner_boundary}
A Menger convex aligned metric space has trivial magnitude homology
in positive degree. 
\end{corollary}

In particular, this corollary applies to any convex or open subset of
$\R^N$. 

\begin{definition}
Two metric spaces $X$ and $Y$ are \demph{magnitude homology equivalent} if
there exist maps $\oppairu{X}{Y}$ inducing mutually inverse maps
$\oppairu{\Mh{n}(X)}{\Mh{n}(Y)}$ for all $n \geq 1$.
\end{definition}

We do not require our maps to induce an isomorphism in degree $0$, since
by Remark~\ref{rmk:0_homology}, this would make $X$ and $Y$ isometric.

\begin{example}
\label{eg:convex-equiv}
All nonempty convex Euclidean sets are magnitude homology equivalent, by
Corollary~\ref{cor:empty_inner_boundary}.
\end{example}

\begin{proposition}
\label{propn:retract}
Let $X$ be an aligned metric space and $A$ a retract of $X$ such that
$\rho X \sub A$. Then $A$ and $X$ are magnitude homology equivalent. 
\end{proposition}

\begin{proof}
Let $\iota: A \hookrightarrow X$ denote the inclusion, and choose a
retraction $\pi: X \to A$. Then $\iota\pi(x) = x$ for all $x \in \rho X$,
since $\rho X \sub A$. Hence by Corollary~\ref{cor:parallel-ib}, the map
\[
\iota_* \pi_* = (\iota\pi)_*: \Mh{n}(X) \to \Mh{n}(X)
\]
is the identity for all $n \geq 1$. On the other hand, $\pi\iota = 1_A$, so
$\pi_* \iota_*$ is the identity on $\Mh{n}(A)$. Hence $\pi_*$ and
$\iota_*$ are mutually inverse in positive degree.
\end{proof}

Alternatively, Proposition~\ref{propn:retract} can be derived from
Proposition~\ref{propn:incl} using the easily established fact that any
retract of a metric space $X$ is convex in $X$.

\begin{theorem}
\label{thm:equiv-to-core}
Every \emph{non}convex closed Euclidean set is magnitude homology
equivalent to its core.
\end{theorem}

The core of a convex set is the empty space
(Example~\ref{egs:core}\bref{eg:core-convex}), which is magnitude homology
equivalent only to itself, so the nonconvexity condition cannot be dropped.

\begin{proof}
Let $X \sub \R^N$ be a nonconvex closed set and $A = \core(X)$, which by
Proposition~\ref{propn:core-retract} is a retract of $X$. Then apply
Proposition~\ref{propn:retract}.
\end{proof}

Since the core construction is idempotent
(Lemma~\ref{lemma:core-elem}\bref{part:ce-idem}),
Theorem~\ref{thm:equiv-to-core} provides a canonical representative of each
magnitude homology equivalence class of nonconvex closed Euclidean sets.

Proposition~\ref{propn:retract} generates many examples of magnitude
homology equivalence:

\begin{proposition}
\label{propn:remove-interior}
Let $S \sub \R^N$. For nonempty convex closed sets $C \sub \R^N$
containing $S$, the magnitude homology equivalence class of $C \setminus
S^\circ$ is independent of the choice of $C$.
\end{proposition}

\begin{proof}
We show that for every such $C$, the space $C \setminus S^\circ$ is
magnitude homology equivalent to $X = \R^N \setminus S^\circ$. To do this,
we apply Proposition~\ref{propn:retract} with $A = C \setminus S^\circ$. It
remains to verify that $\rho X \sub A$ and that $A$ is a retract of $X$.

First, $\rho X \sub \partial X = \partial S^\circ \sub C$, using
Remark~\ref{rmk:bdys}. Hence $\rho X \sub C \cap X = A$.

Now we show that $A$ is a retract of $X$. Let $\pi$ denote metric
projection $\R^N \to C$. It is enough to prove that $\pi X \sub A$, so let
$x \in X$. If $x \in C$ then $\pi(x) = x \in C \cap X = A$. If $x \not\in
C$ then $\pi(x) \in \partial C$ (a general property of metric
projections), which since $S \sub C$ implies that $\pi(x) \in C \setminus
S^\circ = A$. In either case, $\pi(x) \in A$, as required.
\end{proof}

Figure~\ref{fig:equiv_spaces} shows an example of
Proposition~\ref{propn:remove-interior}.

\begin{figure}
    \centering
    \def\svgwidth{\columnwidth}
    %% Creator: Inkscape 1.2.2 (732a01da63, 2022-12-09), www.inkscape.org
%% PDF/EPS/PS + LaTeX output extension by Johan Engelen, 2010
%% Accompanies image file 'equivalent_spaces.pdf' (pdf, eps, ps)
%%
%% To include the image in your LaTeX document, write
%%   \input{<filename>.pdf_tex}
%%  instead of
%%   \includegraphics{<filename>.pdf}
%% To scale the image, write
%%   \def\svgwidth{<desired width>}
%%   \input{<filename>.pdf_tex}
%%  instead of
%%   \includegraphics[width=<desired width>]{<filename>.pdf}
%%
%% Images with a different path to the parent latex file can
%% be accessed with the `import' package (which may need to be
%% installed) using
%%   \usepackage{import}
%% in the preamble, and then including the image with
%%   \import{<path to file>}{<filename>.pdf_tex}
%% Alternatively, one can specify
%%   \graphicspath{{<path to file>/}}
%% 
%% For more information, please see info/svg-inkscape on CTAN:
%%   http://tug.ctan.org/tex-archive/info/svg-inkscape
%%
\begingroup%
  \makeatletter%
  \providecommand\color[2][]{%
    \errmessage{(Inkscape) Color is used for the text in Inkscape, but the package 'color.sty' is not loaded}%
    \renewcommand\color[2][]{}%
  }%
  \providecommand\transparent[1]{%
    \errmessage{(Inkscape) Transparency is used (non-zero) for the text in Inkscape, but the package 'transparent.sty' is not loaded}%
    \renewcommand\transparent[1]{}%
  }%
  \providecommand\rotatebox[2]{#2}%
  \newcommand*\fsize{\dimexpr\f@size pt\relax}%
  \newcommand*\lineheight[1]{\fontsize{\fsize}{#1\fsize}\selectfont}%
  \ifx\svgwidth\undefined%
    \setlength{\unitlength}{675bp}%
    \ifx\svgscale\undefined%
      \relax%
    \else%
      \setlength{\unitlength}{\unitlength * \real{\svgscale}}%
    \fi%
  \else%
    \setlength{\unitlength}{\svgwidth}%
  \fi%
  \global\let\svgwidth\undefined%
  \global\let\svgscale\undefined%
  \makeatother%
  \begin{picture}(1,0.33333333)%
    \lineheight{1}%
    \setlength\tabcolsep{0pt}%
    \put(0,0){\includegraphics[width=\unitlength,page=1]{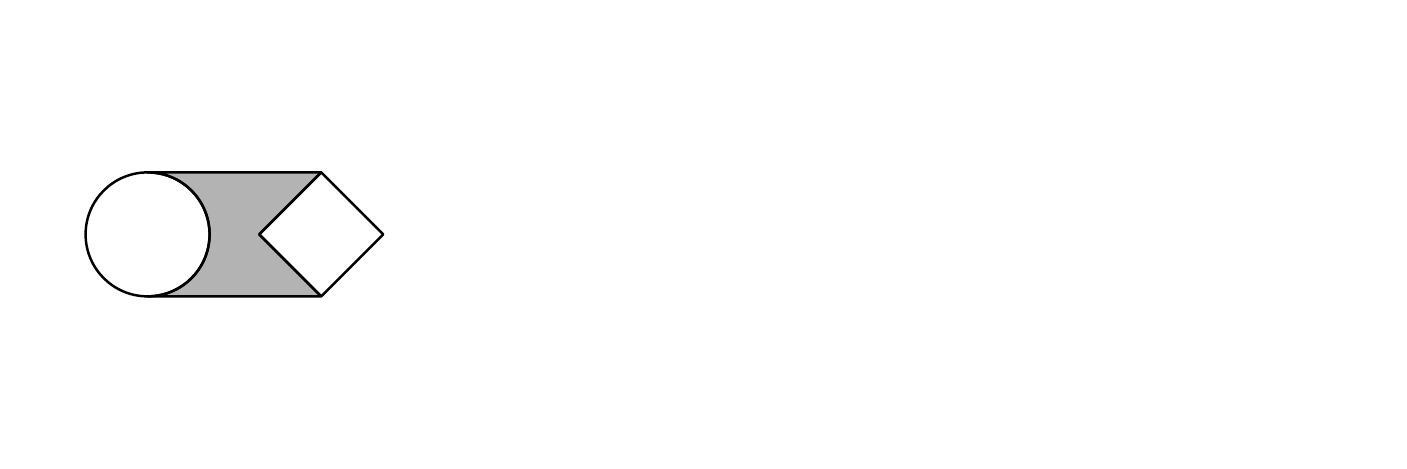}}%
    \put(0.16660882,0.02222221){\color[rgb]{0,0,0}\makebox(0,0)[t]{\lineheight{1.25}\smash{\begin{tabular}[t]{c}(a)\end{tabular}}}}%
    \put(0,0){\includegraphics[width=\unitlength,page=2]{equivalent_spaces.pdf}}%
    \put(0.46055156,0.02222221){\color[rgb]{0,0,0}\makebox(0,0)[t]{\lineheight{1.25}\smash{\begin{tabular}[t]{c}(b)\end{tabular}}}}%
    \put(0.81105939,0.02222221){\color[rgb]{0,0,0}\makebox(0,0)[t]{\lineheight{1.25}\smash{\begin{tabular}[t]{c}(c)\end{tabular}}}}%
  \end{picture}%
\endgroup%

    \caption{Three magnitude homology equivalent spaces, where the set $S$
    of Proposition~\ref{propn:remove-interior} is the union of a disc and a
    filled square. In~(a), $C = \cconv(S)$; in~(b), $C$ is a filled ellipse;
    in~(c), $C = \R^2$.}
    \label{fig:equiv_spaces}
\end{figure}

\begin{corollary}
For $\emptyset \neq S \sub \R^N$, the space $\cconv(S) \setminus
S^\circ$ is magnitude homology equivalent to $\R^N \setminus S^\circ$.
\end{corollary}

This follows immediately from Proposition~\ref{propn:remove-interior}, and
implies in turn:

\begin{corollary}
The boundary $\partial C$ of a nonempty convex set $C \sub \R^N$ is
magnitude homology equivalent to $\R^N \setminus C^\circ$.
\end{corollary}

\section{The main theorem}
\label{sec:main}

For our main theorem, recall that a closed Euclidean set is a metric space
isometric to a closed subset of $\R^N$ for some $N \geq 0$, and that maps
of metric spaces are taken to be short
(Definition~\ref{defn:map}).

\begin{theorem}
\label{thm:main}
Let $X$ and $Y$ be nonempty closed Euclidean sets. The following are
equivalent:
\begin{enumerate}
\item 
\label{part:main-all}
$X$ and $Y$ are magnitude homology equivalent; that is, there exist maps
$\oppair{X}{Y}{f}{g}$ such that
$\oppair{\Mh{n}(X)}{\Mh{n}(Y)}{f_*}{g_*}$ are mutually inverse for
all $n \geq 1$;

\item
\label{part:main-some}
there exist maps $\oppair{X}{Y}{f}{g}$ such that
$\oppair{\Mh{n}(X)}{\Mh{n}(Y)}{f_*}{g_*}$ are mutually inverse for
some $n \geq 1$;

\item 
\label{part:main-ib}
there exist maps $\oppair{X}{Y}{f}{g}$ restricting to mutually inverse
isometries $\oppairu{\rho X}{\rho Y}$;

\item 
\label{part:main-core-maps}
there exist maps $\oppair{X}{Y}{f}{g}$ restricting to mutually inverse
isometries $\oppairu{\core(X)}{\core(Y)}$;

\item
\label{part:main-core}
$\core(X)$ and $\core(Y)$ are isometric.
\end{enumerate}
\end{theorem}

Most importantly, magnitude homology equivalence~\bref{part:main-all} is
equivalent to the concrete geometric condition~\bref{part:main-core}. 

\begin{proof}
Conditions \bref{part:main-all}--\bref{part:main-ib} are equivalent for all
aligned spaces, by Theorem~\ref{thm:op-tfae}.
\bref{part:main-ib}$\implies$\bref{part:main-core-maps} follows from
Lemma~\ref{lemma:ib-to-core}, and
\bref{part:main-core-maps}$\implies$\bref{part:main-core} is trivial.

\bref{part:main-core}$\implies$\bref{part:main-all} follows from
Theorem~\ref{thm:equiv-to-core} in the case where $X$ and $Y$ are both
nonconvex. By Example~\ref{egs:core}\bref{eg:core-convex}, the only other
possibility is that $X$ and $Y$ are both convex, in which
case~\bref{part:main-all} holds by Example~\ref{eg:convex-equiv}.
\end{proof}

\begin{remarks}
\begin{enumerate}
\item
The equivalent conditions of Theorem~\ref{thm:main} are strictly stronger
than the condition that $X$ and $Y$ have isometric inner boundaries, by
Remark~\ref{rmks:opposing-maps}\bref{rmk:om-ib}.

\item 
The equivalent conditions of Theorem~\ref{thm:main} are also strictly
stronger than the condition that there is a map $f: X \to Y$ such that
$f_*: \Mh{n}(X) \to \Mh{n}(Y)$ is an isomorphism for all $n \geq
1$. (And this in turn is stronger than the condition that $\Mh{n}(X)
\cong \Mh{n}(Y)$ for all $n \geq 1$:
Remark~\ref{rmks:opposing-maps}\bref{rmk:om-circles}.) 

Indeed, let $X = \{0\} \cup [1, 2] \cup \{3\}$ and $Y = \{0, 1, 2\}$, and
define $f: X \to Y$ by $f(0) = 0$, $f[1, 2] = \{1\}$ and $f(3) = 2$. For $n
\geq 1$, the thin chains in $X$ and $Y$ are given by
\begin{align*}
\THIN{n}{n}(X)  &
=
\{ (0, 1, 0, 1, \ldots), (1, 0, 1, 0, \ldots),
(2, 3, 2, 3, \ldots), (3, 2, 3, 2, \ldots) \},  \\
\THIN{n}{n}(Y)  &
=
\{ (0, 1, 0, 1, \ldots), (1, 0, 1, 0, \ldots),
(1, 2, 1, 2, \ldots), (2, 1, 2, 1, \ldots) \},  
\end{align*}
and $\THIN{n}{\ell}(X) = \emptyset = \THIN{n}{\ell}(Y)$ when $\ell \neq
n$. In all cases, the map $f_\star: \Z\THIN{n}{\ell}(X) \to
\Z\THIN{n}{\ell}(Y)$ is a bijection, so $f_*: \MH{n}{\ell}(X) \to
\MH{n}{\ell}(Y)$ is an isomorphism. However, $\core(X) = X$ and $\core(Y) =
Y$, which are not isometric.

In other words, magnitude homology equivalence is a stronger property than
quasi-isomorphism in positive degree. Finding a concrete geometric
description of quasi-isomorphism for magnitude homology remains an open
question. 
\end{enumerate}
\end{remarks}

\paragraph{Acknowledgements} The first author was funded by a PhD
scholarship from the Carnegie Trust for the Universities of Scotland.

\bibliography{magnitude}

\end{document}